\documentclass[preprint,12pt]{elsarticle}

\usepackage{amssymb}
\usepackage{amsmath}
\usepackage{amsthm}
\usepackage{array}
\usepackage{rotating}
\usepackage{vmargin}
\setmarginsrb{2.2cm}{1cm}{2.2cm}{3cm}{1cm}{1cm}{1.5cm}{1.5cm}
\newtheorem{thm}{Theorem}

\newtheorem{prop}{Proposition}

\newtheorem{lem}[thm]{Lemma}
\newdefinition{rem}{Remark}
\newdefinition{defi}{Definition}
\newcolumntype{M}[1]{>{\raggedright}m{#1}}

\journal{*****}

\begin{document}

\begin{frontmatter}

\author{Thierry COMBOT\fnref{label2}}
\ead{thierry.combot@u-bourgogne.fr}
\address{IMB, Universi\'e de Bourgogne, 9 avenue Alain Savary, 21078 Dijon Cedex }

\title{Reduction of symbolic first integrals of planar vector fields}

\author{}

\address{}

\begin{abstract}
Consider a planar polynomial vector field $X$, and assume it admits a symbolic first integral $\mathcal{F}$, i.e. of the $4$ classes, in growing complexity: Rational, Darbouxian, Liouvillian and Riccati. If $\mathcal{F}$ is not rational, it is sometimes possible to reduce it to a simpler class first integral. We will present algorithms to reduce symbolic first integral to a lower complexity class. These algorithms allow to find the minimal class first integral and in particular to test the existence of a rational first integral except in the case where $\mathcal{F}$ is a $k$-Darbouxian first integral without singularities and $k\in\{2,3,4,6\}$. In this case, several examples are built and a procedure is presented which however requires the computation of elliptic factors in the Jacobian of a superelliptic curve.

\end{abstract}
\end{frontmatter}

\section{Introduction}

Consider the differential system
$$(S)\quad \left\lbrace \begin{array}{rcl} \dot{x} &=& A(x,y), \\ \dot{y} &=& B(x,y), \end{array}\right. \qquad A,B\in \mathbb{Q}[x,y],\; A \wedge B=1$$
First integrals are \emph{non-constant functions} $\mathcal{F}$ that are constant along the solutions $\big(x(t),y(t)\big)$ of $(S)$. This property can be rewritten as being a solution of a partial differential equation
\begin{equation}\label{eqint}\tag{Eq}
A(x,y) \partial_x \mathcal{F}(x,y)+B(x,y) \partial_y \mathcal{F}(x,y)=0,
\end{equation}
Noting $D=A\partial_x+B\partial_y$, the first integrals are such that $D(\mathcal{F})=0$.

As given in \cite{cheze2020symbolic}, we will consider the $4$ classes of symbolic first integrals $\mathcal{F}$ of $(S)$, which are
\begin{itemize}
\item \emph{rational} if there exists $F\in\overline{\mathbb{Q}}(x,y)\setminus \overline{\mathbb{Q}}$ such that 
\begin{equation}\label{eqtype0}\tag{Rat}
\mathcal{F}-F(x,y)=0.
\end{equation}
\item \emph{$k$-Darbouxian} if there exists $F(x,y)$ such that 
\begin{equation}\label{eqtype1}\tag{D}
\partial_y \mathcal{F}-F(x,y) =0,\qquad F^k\in\overline{\mathbb{Q}}(x,y)\setminus \{0\}.
\end{equation}
\item \emph{Liouvillian}  if there exists $F\in\overline{\mathbb{Q}}(x,y)$ such that 
\begin{equation}\label{eqtype2}\tag{L}
\partial_y^2 \mathcal{F}-F(x,y)\partial_y \mathcal{F}=0.
\end{equation}
\item \emph{Riccati} if  $\mathcal{F}$ is the 
quotient of two independent solutions over $\overline{\mathbb{Q}(x)}$  of
\begin{equation}\label{eqtype3}\tag{Ric}
\partial_y^2 \mathcal{F}-F(x,y) \mathcal{F}=0,\qquad F\in\overline{\mathbb{Q}}(x,y).
\end{equation}
\end{itemize}
Recall that when $(S)$ admits such first integrals with an associated $F\in \overline{\mathbb{Q}}(x,y)$, then in fact it admits one with $F\in \mathbb{Q}(x,y)$. The algorithms in \cite{cheze2020symbolic} allow to find such first integrals up to some degree bound on $F$ given by the user. We order them in growing class
$$\hbox{Rational} < \hbox{Darbouxian} <\hbox{Liouvillian} <\hbox{Riccati}$$
However, nothing guarantees that the first integral found by \cite{cheze2020symbolic} will be in the simplest class possible. In fact, it is typical that a lower class first integral exists, but of much higher degree, and thus was not detected as it exceeded the bound given by the user.

The reduction problem consists of, given a first integral, find a lower class first integral or prove it does not exist. As we will see, the particular case of testing if a rational first integral exists is the most difficult case. It is a sub problem of the Poincare problem \cite{pereira2002poincare}, which look for an algorithm for finding rational first integral of polynomial vector fields. In our case, the knowledge of a (transcendental) first integral will be put to profit to solve the question of existence of a rational first integral, except in a few resistant cases.

\begin{thm}\label{thm2}
Consider a polynomial vector field $X=(A,B),\; A,B\in\mathbb{Q}[x,y]$ which admits a Riccati first integral $\mathcal{F}$, defined by its associated $F\in \mathbb{Q}(x,y)$. Then algorithm \underline{\sf ReduceRiccati} returns a lower class first integral if and only if such integral exists.
\end{thm}

\begin{thm}\label{thm1}
Consider a polynomial vector field $X=(A,B),\; A,B\in\mathbb{Q}[x,y]$ which admits a Liouvillian first integral $\mathcal{F}$, defined by its associated $F\in \mathbb{Q}(x,y)$. Then algorithm \underline{\sf ReduceLiouvillian} returns a lower class first integral if and only if such integral exists.
\end{thm}

\begin{thm}\label{thm3}
Consider a polynomial vector field $X=(A,B),\; A,B\in\mathbb{Q}[x,y]$ which admits a $k$-Darbouxian first integral $\mathcal{F}$, defined by its associated $F\in \mathbb{Q}(x,y)$. Assume that $k\neq 2,3,4,6$, or $\mathcal{F}$ admits at least a singularity. Then algorithm \underline{\sf ReduceDarbouxian} returns a rational first integral if and only if such integral exists.
\end{thm}

Applying sequentially these algorithms, we can find the first integral of minimal class of $X$ except when encountering in Theorem \ref{thm3} the forbidden case. The forbidden case could be solved if we could compute the Mordell Weil group of an elliptic threefold \cite{hulek2011calculating}. An example of Lins Neto \cite{neto2002some} happens to be in this case, and we will reduce it following this approach, alongside others. In section $2$, we will prove Theorem \ref{thm1}, in section $3$ Theorem \ref{thm2}, in section $4$ Theorem \ref{thm3} and in section $5$ present a procedure and examples falling in the unsolved case.\\

Notations: We note $\hbox{num}$ the numerator of a rational fraction, $\hbox{den}$ the denominator, and $\hbox{sf}$ the square free part of a polynomial, the polynomial with factor multiplicity removed.

\section{Reduction of Liouvillian first integrals}

Recall that a Liouvillian first integral $\mathcal{F}$ is defined up to affine transformation by
$$\frac{\partial_y^2 \mathcal{F}}{\partial_y \mathcal{F}}=F.$$
We call $\mathcal{R}=\partial_y \mathcal{F}$ the integrating factor. It is a hyperexponential function, i.e. a function whose logarithmic differential is a rational $1$ form. We have
$$\frac{\partial_x \mathcal{R}}{\mathcal{R}}=\frac{\partial_x\partial_y \mathcal{F}}{\partial_y \mathcal{F}}=\frac{\partial_y \left(-\frac{B}{A} \partial_y \mathcal{F}\right)}{\partial_y \mathcal{F}}=\frac{\partial_y \left(-\frac{B}{A}\right) \partial_y \mathcal{F}-\frac{B}{A} \partial_y^2 \mathcal{F} }{\partial_y \mathcal{F}}=-\partial_y \left(\frac{B}{A}\right)-\frac{B}{A} F$$
and we note this rational function $G$. Thus we have $\frac{d\mathcal{R}}{\mathcal{R}}=Gdx+Fdy$ and then the first integral $\mathcal{F}$ can be written
$$\mathcal{F}(x,y)= \int e^{\int G dx+Fdy} \left(-\frac{B}{A}dx+dy\right)$$
If $X$ also admits another first integral $J$, then $\mathcal{F}$ is a function of $J$, and thus
$$\mathcal{F}(x,y)=f(J(x,y)),\quad f \hbox{ one variable function}$$
The function $f$ can be multivalued, but is locally analytic except at finitely many points which are the image by $J$ of the singular points of $\mathcal{F}$.\\

\noindent\underline{\sf ReduceLiouvillian}\\
\textsf{Input:} Vector field $X=(A,B),\; A,B\in\mathbb{Q}[x,y],\; A\wedge B=1$, and $F\in\mathbb{Q}(x,y)$ corresponding to a Liouvillian first integral of $X$. \\
\textsf{Output:} An equation defining a rational or $k$-Darbouxian first integral, or ``None''
\begin{enumerate}
\item Note $G=-\partial_y \left(\frac{B}{A}\right)-\frac{B}{A} F$, $\tilde{Q}=\hbox{lcm}(\hbox{den}(F),\hbox{den}(G))$, $Q=\tilde{Q}/\hbox{sf}(\tilde{Q})$. Search a non constant solution of the linear system
$$D\left(\frac{P}{Q}\right)=0$$
with $P\in\mathbb{Q}[x,y],\; \deg P \leq 1+\max (\deg F\tilde{Q}, \deg G\tilde{Q} ) -\deg \hbox{sf}(\tilde{Q})$. If one, return $\mathcal{F}-\tfrac{P}{Q}=0$.
\item Note $Q=\hbox{sf}(\hbox{den}(F))$. Search a non zero solution of the linear system
$$D\left(\frac{P}{Q}\right)+A\partial_y\left(\frac{B}{A}\right) \frac{P}{Q}=0$$
for $P\in\mathbb{Q}[x,y],\; \deg P \leq \deg Q -1$. If one, return $\partial_y \mathcal{F}-\tfrac{P}{Q}=0$.
\item Note $\tilde{F}=F+\epsilon G$. If $\tilde{F}$ has only simple poles, $\deg \hbox{num}(\tilde{F}) < \deg \hbox{den}(\tilde{F})$, and
$$S=\hbox{resultant}(\hbox{num}(\tilde{F})-\lambda (\partial_y \hbox{den}(\tilde{F})+\epsilon \partial_x \hbox{den}(\tilde{F})) , \hbox{den}(\tilde{F}) ,y)$$
has only rational solutions with $k$ the lcm of their denominators, return
$$\partial_y \mathcal{F}-\left(\prod\limits_{\lambda\in S^{-1}(0)} \hbox{gcd}(\hbox{num}(\tilde{F})-\lambda (\partial_y \hbox{den}(\tilde{F})+\epsilon \partial_x \hbox{den}(\tilde{F})), \hbox{den}(\tilde{F}))^{k\lambda}\right)^{1/k} $$
\item Look for rational solutions of the PDE system
$$A\partial_x W +B\partial_y W+(AG+BF)W=0,\quad \partial_y^2 W+F\partial_y W +W \partial_y F=0$$
If a rational solution $W$ exists, return $\partial_y \mathcal{F}-(F+\tfrac{\partial_y W}{W} ) $
\item Return ``None''
\end{enumerate}

\begin{proof}[Proof of Theorem \ref{thm1}]
Assume $X$ admits another first integral $J$ of a lower class, so $k$-Darbouxian or rational. We can then write
$$\mathcal{F}(x,y)=f(J(x,y)),\quad f \hbox{ one variable function}$$
We differentiate both sides giving
$$\partial_y \mathcal{F}(x,y)= \partial_y J(x,y) f'(J(x,y))$$
$$\partial_y^2 \mathcal{F}(x,y)= \partial_y^2 J(x,y) f'(J(x,y))+(\partial_y J(x,y))^2 f''(J(x,y))$$
Making the quotient, we get
$$F(x,y)= \frac{\partial_y^2 J(x,y)}{\partial_y J(x,y)} +(\partial_y J(x,y)) \frac{f''}{f'}(J(x,y))$$
Let us first assume that $X$ admits no rational first integral and thus $J$ is $k$-Darbouxian. Consider the equation
$$\frac{F(x,y)- \frac{\partial_y^2 J(x,y)}{\partial_y J(x,y)}}{\partial_y J(x,y)}=h.$$
If the left hand side is constant, then $\frac{f''}{f'}=h$ for some $h\in\mathbb{C}$ as a function. If the left hand side is not constant, this defines an algebraic curve $\mathcal{C}_h$. As $X$ does not admit a rational first integral, it has at most finitely many Darboux polynomials, and thus there exists $h\in\mathbb{C}$ such that $J_{\mid \mathcal{C}_h}$ is not constant. But then
$$\frac{f''}{f'}(J_{\mid \mathcal{C}_h})=h$$
and thus again $\frac{f''}{f'}=h$ as a function. Solving this differential equation, we conclude that either $f(z)=az+b$ or $f(z)=ae^{bz}+c$ with $a,b,c\in\mathbb{C}$.
The first case implies that $\mathcal{F}$ itself should be a $k$-Darbouxian first integral, and thus that $ e^{\int Gdx+Fdy}$ should be algebraic. This implies that $\int Fdy+Gdx$ is a rational linear combination of logs, and thus 
$Gdx+Fdy$ should have only simple poles. Considering this $1$-form restricted to a line of slope $\epsilon$, we follow Trager \cite{trager1984integration} which gives that the residues are solutions of the polynomial $S$ computed in step $3$. For a generic $\epsilon$, no residues will be missed as the line of slope $\epsilon$ will intersect transversally all poles of $Gdx+Fdy$. Also these residues will not depend on $\epsilon$ as the form $Gdx+Fdy$ is closed. We then test if these residues are rational, and then return the $k$-Darbouxian first integral in step $3$.

The second case implies that $\mathcal{F}$ is (up to affine transformation) the exponential of a Darbouxian first integral. The only possibility is then that $ \int e^{\int Gdx+Fdy} (-\frac{B}{A}dx+dy)$ integrates using a hyperexponential function. But then such a solution can be written $We^{\int Fdy+Gdx}$ with $W$ rational. We can now plug this expression in the two equations that must be satisfied by $\mathcal{F}$
$$\partial_y^2 \mathcal{F}-F\partial_y \mathcal{F}=0,\quad A\partial_x\mathcal{F}+B\partial_y \mathcal{F}=0$$
which gives the system
$$A\partial_x W +B\partial_y W+(AG+BF)W=0,\quad \partial_y^2 W+F\partial_y W +W \partial_y F=0$$
Such system can be solved in rational functions as an integrable connection \cite{barkatou2012computing}. It has at most a vector space of dimension $2$ of rational solutions. If it was $2$, remembering that $\mathcal{F}=1$ is a solution of the previous system, we would have a rational $W$ such that $We^{\int Fdy+Gdx}=1$. Such case with a rational integrating factor would have already been detected in step $3$ as then all roots of $S$ would be integers. Thus a most one solution up to constant factor exists. Step $4$ then returns the equation for the Darbouxian first integral $\ln W +\int Fdy+Gdx$.

Let us now assume $X$ admits a rational first integral. We have
$$\frac{f''}{f'}(J(x,y))=\frac{F(x,y)- \frac{\partial_y^2 J(x,y)}{\partial_y J(x,y)}}{\partial_y J(x,y)} \in \mathbb{C}(x,y)$$
and thus the right-hand side is a function of $J$ only. We can assume that $J$ is indecomposable, and then right-hand side is a rational function of $J$ (see \cite{cheze2014decomposition}). Thus $\frac{f''}{f'}\in\mathbb{C}(z)$. Solving this equation gives
$$f(z)= \int e^{\int g(z) dz} dz,\;\; g\in\mathbb{C}(z)$$
If $e^{\int g(z) dz}$ admits a ramification point $z_0$ with irrational exponent, as the sum of the residues of $g$ over $\mathbb{P}^1$ should be zero, there exists another point $z_1$ with irrational ramification exponent. Composing $f$ with a Moebius transformation, we can assume $z_0=0,z_1=\infty$.
Thus the curves $\hbox{num}(J)=0$ and $\hbox{den}(J)=0$ are ramification curves with irrational exponents of $\mathcal{F}$. Noting
$$\int Gdx+Fdy= F_0(x,y)+\sum \lambda_i \ln F_i(x,y),$$
the ramification curves of $\mathcal{F}$ are factors of the $F_i$. Thus the factors of $J$ are the factors of the $F_i$, and thus the poles of $\partial_y \ln J$ are factors of the $F_i$. These factors of the $F_i$ also appear as poles of the differential form $Gdx+Fdy$. The factors of $J$ depending only on $x$ do not give poles to $\partial_y \ln J$ due to the $y$ differentiation, and thus the poles of $\partial_y \ln J$ are factors of $\hbox{den}(F)$. They are all of order $1$, thus the denominator of $\partial_y \ln J$ divides $\hbox{sf}(\hbox{den}(F))$, defined as $Q$ in step $2$. The degree of the numerator of $\partial_y \ln J$ is less than its denominator by construction, thus $\partial_y \ln J$ can be written $\frac{P}{Q}$ with $\deg P \leq \deg Q -1$. In \cite{cheze2020symbolic}, it is proved that a rational function $P/Q$ is the $y$ derivative of a Darbouxian first integral if and only if it satisfies the equation
$$D\left(\frac{P}{Q}\right)+A\partial_y\left(\frac{B}{A}\right) \frac{P}{Q}=0$$
This is the equation solved in step $2$, and thus step $2$ would return the equation for the Darbouxian integral $\ln J$.

Else $e^{\int g(z) dz}$ admits at least one essential singularity (as if all singularities were regular with rational exponent, then $e^{\int g(z) dz}$ would be algebraic, and this would be detected in step $3$). We can write
$$e^{\int g(z) dz}=K(z)^{1/k}e^{L(z)},\quad K,L\in\mathbb{C}(z),\; k\in \mathbb{N}^*$$
So let us differentiate the equality
$$\int^J(x,y) K(z)^{1/k}e^{L(z)} dz=\mathcal{F}(x,y)$$
$$\partial_y J(x,y) K(J(x,y))^{1/k}e^{L(J(x,y))}=\partial_y \mathcal{F}(x,y)=e^{F_0(x,y)}\prod F_i(x,y)^{\lambda_i}$$
Thus we have $\lambda_i\in \frac{1}{k}\mathbb{Z}$, and we can identify the exponential part $F_0(x,y)=L(J(x,y))$. Thus $F_0$ is also a rational first integral, and the poles of $F_0$ are the poles of the differential form $Gdx+Fdy$ with multiplicity at most one less, thus divide $Q$ as defined in step $1$. The pole along the line at infinity of $F_0$ is at most the order of the pole along the line at infinity of the form $Gdx+Fdy$ plus one, and thus we can write $F_0(J)=P/Q$ with $\deg P \leq \max(\deg F\tilde{Q} ,\deg G\tilde{Q})-\deg \hbox{sf}(\tilde{Q})+1$. Such rational first integral would be found in step $1$.

If everything before fails to find a simpler first integral, then none such exists. Step $5$ then returns ``None''. 
\end{proof}

Remark that the result returned by the algorithm \underline{\sf ReduceLiouvillian} can be a Darbouxian first integral even if a rational first integral exists. However, if it returns ``None'', then it is guaranteed that no simpler than Liouvillian first integral exists.

\begin{prop}
The equation returned by \underline{\sf ReduceLiouvillian} has always rational coefficients.
\end{prop}

\begin{proof}
In the first and second step, a linear system with coefficients in $\mathbb{Q}$ is solved, thus a basis of solutions with coefficients in $\mathbb{Q}$ exists, and one of them is returned. In the third step, the rational function in the differential equation returned is
$$\prod\limits_{\lambda\in S^{-1}(0)} \hbox{gcd}(\hbox{num}(\tilde{F})-\lambda (\partial_y \hbox{den}(\tilde{F})+\epsilon \partial_x \hbox{den}(\tilde{F})), \hbox{den}(\tilde{F}))^{k\lambda}.$$
The polynomial $S$ has only rational solutions, thus the gcd holds over polynomials with rational coefficients, and thus are with rational coefficients. The exponents $k\lambda$ are integers by construction, and thus this is a rational function with rational coefficients. In the fourth step, an integrable connection is solved in rational functions. The coefficients of this integrable connection have rational coefficients, and then there always exists a basis of rational solutions with rational coefficients \cite{barkatou2012computing}. Thus the one returned has rational coefficients. 
\end{proof}

\section{Reduction of Riccati first integrals}

Recall that a Riccati first integral $\mathcal{F}$ is defined up to Moebius transformation as the quotient $\mathcal{F}_1/\mathcal{F}_2$ of a $\mathbb{C}(x)$ basis of solutions of a differential equation of the form
\begin{equation}\label{eq2}
\partial_y^2 \mathcal{F}_i-F\mathcal{F}_i=0
\end{equation}
The condition of $\mathcal{F}$ being a first integral adds in fact additional conditions for the differentiation in $x$. In \cite{cheze2020symbolic}, it is proven that completing this equation to give the system of PDEs
\begin{equation}\label{eq2b}
\partial_y^2 \mathcal{F}_i-F\mathcal{F}_i=0,\quad D(\mathcal{F}_i)=\frac{1}{2} A \partial_y \left( \frac{B}{A} \right)\mathcal{F}_i
\end{equation}
gives a $\mathbb{C}$ basis of solutions of dimension $2$, thus defining up to Moebius transformation the quotient of two solutions. The first step will consist to check if $\mathcal{F}$ itself is Liouvillian. For this we will use a variation of the Kovacic algorithm \cite{kovacic1986algorithm}, following \cite{ulmer1996note}. We will need to compute symmetric powers of \eqref{eq2b}. The $m$th symmetric power of \eqref{eq2b} is a PDE system whose solutions are the $m$th power of solutions of \eqref{eq2b}. In particular, they are solutions of the $m$th symmetric power of the operator $\partial_y^2-F$ where $x$ is seen as a parameter. To complete this equation for the dependence in $x$ of the solutions, we simply remark that 
$$ D(\mathcal{F}_i^m)=\frac{m}{2} A \partial_y \left( \frac{B}{A} \right)\mathcal{F}_i^m$$
and thus this $m$th symmetric power of \eqref{eq2b} is the system
\begin{equation}\label{eq2c}
(\partial_y^2 -F)^{\otimes m} \mathcal{F}_i=0,\quad D(\mathcal{F}_i)=\frac{m}{2} A \partial_y \left( \frac{B}{A} \right)\mathcal{F}_i
\end{equation}
Such system can be written as an integrable connection, i.e. writing $V=(\mathcal{F}_i,\dots, \partial_y^{m-1} \mathcal{F}_i)^\intercal$ and $\partial_y V,\partial_x V$ satisfy the equation system
$$\partial_y V= M_1 V,\; \partial_x V= M_2 V.$$
Then the package IntegrableConnections \cite{barkatou2012computing} can compute the rational and hyperexponential solutions. Remark that the current implementation of this package is slow and faulty for even the simplest examples. Thus for the practical implementation, for rational solutions we first look for a $\mathbb{C}(x)$ basis of rational solutions of $(\partial_y^2 -F)^{\otimes m}$, and then look for rational coefficients in $x$ by writing a system and solving it through LinearFunctionalSystems package of Maple. For hyperexponential solutions, going from a $\mathbb{C}(x)$ basis of hyperexponential solutions of $\partial_y^2 -F$ to a basis of hyperexponential solutions of \eqref{eq2b} can be troublesome. Thus we first make a generic affine transformation on the variables before looking for hyperexponential solutions. On a generic straight line there is a one to one correspondence between hyperexponential solutions on the line and on the whole $x,y$ plane, allowing to transform the system to a single equation with two parameters ($x$ and the slope of the straight line), and solutions are found through expsols of DETools package.\\

More generally we will also need to check if $\mathcal{F}$ is a function (not homographic) of another first integral of a lower class, i.e.
$$\mathcal{F}(x,y)=f(J(x,y)),\quad f \hbox{ one variable function}$$
We will use the Schwarzian derivative
$$S_y(\mathcal{F})=\frac{\partial_y^3 \mathcal{F}}{\partial_y \mathcal{F}}-\frac{3}{2} \left( \frac{\partial_y^2\mathcal{F}}{\partial_y\mathcal{F}}\right)^2$$
When $\mathcal{F}=\mathcal{F}_1/\mathcal{F}_2$ with $\partial_y^2 \mathcal{F}_i-F\mathcal{F}_i=0$, calculations give that
$$S_y(\mathcal{F})=-2F$$
and this gives an alternative definition of $\mathcal{F}$ in function of $F$ (see \cite{casale2004groupoide}).\\

\noindent\underline{\sf ReduceRiccati}\\
\textsf{Input:} Vector field $X=(A,B),\; A,B\in\mathbb{Q}[x,y],\; A\wedge B=1$, and $F\in\mathbb{Q}(x,y)$ corresponding to a Riccati first integral of $X$.\\ 
\textsf{Output:} An equation defining a rational, $k$-Darbouxian or Liouvillian first integral, or ``None''.
\begin{enumerate}
\item Search a non zero rational solution $R$ of \eqref{eq2b}. If one, return $\partial_y \mathcal{F} -R^{-2}$.
\item Search a non zero rational solution $R$ of the symmetric square of \eqref{eq2b}. If one, compute the special polynomial $P(u)=u^2+b_1u+b_0$ attached to it.
\begin{enumerate}
\item If $\Delta(P)=0$ return $\partial_y^2 \mathcal{F}-b_1 \partial_y \mathcal{F}$
\item If $\Delta(P)\neq 0$ return $\partial_y \mathcal{F}-\sqrt{\Delta(P)/c}$ where $c$ is the dominant coefficient of $\Delta(P)$ for some monomial order.
\end{enumerate}
\item Search a hyperexponential solution $R$ of \eqref{eq2b}. If one, return $\partial_y^2 \mathcal{F} +2\frac{\partial_y R}{R} \partial_y \mathcal{F}$.
\item Search a non zero rational solution $R$ of the fourth symmetric power of \eqref{eq2b}. If exactly one, compute the special polynomial $P(u)^2=(u^2+b_1u+b_0)^2$ attached to it. Return $\partial_y \mathcal{F}-\sqrt{\Delta(P)}$
\item For $m\in\{6,8,12\}$, search a non zero rational solution $R$ of the $m$th symmetric power of \eqref{eq2b}. If one, note $\partial_y^2+a_1\partial_y +a_0= (\partial_y^2-F) \otimes \left(\partial_y +\frac{\partial_y R}{mR} \right)$ and return $\mathcal{F}-a_0^{-1}(2a_1+\frac{\partial_y a_0}{a_0})^2$.
\item Note $Q=\hbox{sf}(\hbox{den}(F))$. Solve the linear system
$$D\left(\frac{P}{Q}\right)+A\partial_y\left(\frac{B}{A}\right) \frac{P}{Q}=0$$
for $P\in\mathbb{Q}[x,y],\; \deg P \leq \deg Q-1$. If a non zero solution exists, return $\partial_y \mathcal{F}-\frac{P}{Q}=0$.
\item Note $Q=\hbox{sf}(\hbox{den}(F)A)$. Look for rational first integrals $R$ of $X$ of numerator and denominator degree $\leq \max(6(\deg Q-1),\deg \hbox{num}(F)+2,\deg \hbox{den}(F)+1)$. If one $R$, return $\mathcal{F}-R$.
\item Return ``None''.
\end{enumerate}

Let us recall some definitions about linear differential equation of order $2$ (see \cite{van2012galois})
\begin{defi}
A differential equation $f''(z)+a(z)f'(z)+b(z)$ with $a,b\in\mathbb{C}(z)$ has a singularity at $z_0$ when $z_0$ is a pole of $a$ or $b$. It has a singularity at infinity if after variable change $z\rightarrow 1/z$, it has a singularity at $0$.\\
A singularity $z_0$ is said to be regular if there exists a basis of series solutions of the form  $(z-z_0)^\alpha \ln (z-z_0)^r \mathbb{C}[[z-z_0]]$ with $\alpha\in\mathbb{C}, r\in\mathbb{N}$. Else it is said to be irregular.\\
The numbers $\alpha$ of the basis are called the exponents.\\
A singularity $z_0$ is said to be meromorphic if there exists a basis of solutions of Laurent series.\\
A singularity $z_0$ is said to be algebraic if there exists a basis of solutions of Puiseux series.\\
\end{defi}

\begin{proof}[Proof of Theorem \ref{thm2}]
If $\mathcal{F}$ is itself Liouvillian, then obviously the system admits a Liouvillian first integral. Let us first prove that if $\mathcal{F}$ is Liouvillian, then steps $1-5$ will return an equation for a first integral of a lower class than Riccati.

\begin{lem}\label{lem0}
If $\mathcal{F}$ is Liouvillian, then an equation for a Liouvillian, Darbouxian or rational first integral is returned in \underline{\sf ReduceRiccati} in steps $1-5$.
\end{lem}

\begin{proof}
If $\mathcal{F}$ is Liouvillian, then we have the relation $\mathcal{F}_2=\mathcal{F} \mathcal{F}_1$, which inserted in the formula of the Wronskian (which is equal to $1$ here) gives a first order differential equation for $\mathcal{F}_1$ with Liouvillian coefficients. Thus $\mathcal{F}_1$ and then $\mathcal{F}_2$ are Liouvillian, and thus the Galois group of \eqref{eq2b} is virtually solvable.

We now follow the algorithm of \cite{ulmer1996note}. Step $1$ looks for rational solutions $R$. If one is found, then the other one is $R\int R^{-2} dy$ which gives a Darbouxian first integral $\int R^{-2} dy$. Step $2$ looks for rational solutions of the symmetric square of \eqref{eq2b}. Note $P(u)$ the special polynomial. Any roots of $P(u)$ are logarithmic derivatives of solutions of \eqref{eq2b}. If $\Delta=0$, this gives one hyperexponential solution $H$, and the other solution can be written $H \int H^{-2} dy$, which gives a Liouvillian first integral $\int H^{-2} dy$. If $P$ has two distinct roots, this gives two solutions of the form $\exp(\frac{1}{2} \int T\pm \sqrt{\Delta})$. Their quotient gives a first integral $\exp(\int \sqrt{\Delta})$ and thus a $2$-Darbouxian first integral exists, given by equation $\partial \mathcal{F}-\sqrt{\Delta/c}$.

Step $3$ looks for a hyperexponential solution $H$, and as in step $2$ if one is found, we have a Liouvillian first integral  $\int H^{-2} dy$. Step $4$ looks for the dihedral Galois group case. If a solution $R$ is found of the fourth symmetric power, then building the special polynomial $P(u)^2=(u^2+b_1u+b_0)^2$, the system admits solutions of the form
$$e^{\int -b_1+\sqrt{\Delta} dy},  e^{\int -b_1-\sqrt{\Delta} dy}$$
and thus the quotient of them gives the formula for $\mathcal{F}=\exp \int 2\sqrt{\Delta} dy$. Thus $\int \sqrt{\Delta} dy$ is a $2$-Darbouxian first integral, and it is returned in step $4$. Step $5$ looks for finite primitive Galois groups cases. For this part, we follow \cite{van2005solving} instead. If a solution $R$ is found, then a basis of solutions of the form
$$R^{1/m} \mathcal{H}_1(f),R^{1/m} \mathcal{H}_2(f)$$
exists with
\begin{itemize}
\item $m=6$, $f=\sqrt{1+\frac{64}{5} \frac{a_0}{g^2}}$
\item $m=8$, $f=-\frac{7}{144} \frac{g^2}{a_0}$
\item $m=12$,$f=\frac{11}{400} \frac{g^2}{a_0}$
\end{itemize}
$g=2a_1+\frac{\partial_y a_0}{a_0}$ and $\mathcal{H}_1,\mathcal{H}_2$ form a basis of one of the three standard equations (see \cite{van2005solving}). When making the quotient of these two solutions, we obtain that $\frac{\mathcal{H}_1}{\mathcal{H}_2}(f)$ is a first integral, and thus so is $f$. Looking at the formulas, they all are functions of $g^2/a_0$, and thus $g^2/a_0$ is in any case a rational first integral. This is the expression returned in step $5$. 
If none of these steps returned a lower class first integral, then the Galois group of \eqref{eq2b} is $SL_2(\mathbb{C})$, thus not virtually solvable.
\end{proof}

Assume $X$ admits another first integral $J$, which can be Liouvillian, Darbouxian or rational. We can then write
$$\mathcal{F}(x,y)=f(J(x,y)),\quad f \hbox{ one variable function}$$
We apply the Schwarzian derivative both sides giving
$$-2F(x,y)=(\partial_y J(x,y))^2 \left(\frac{f'''}{f''}-\frac{3}{2}\left(\frac{f''}{f'}\right)^2\right)(J(x,y))+
\frac{\partial_y^3 J(x,y)}{\partial_y J(x,y)}-\frac{3}{2} \left( \frac{\partial_y^2J(x,y)}{\partial_yJ(x,y)}\right)^2$$
and thus we have
\begin{equation}\label{eq1}
\left(\frac{f'''}{f''}-\frac{3}{2}\left(\frac{f''}{f'}\right)^2\right)(J(x,y))=
-2\frac{F(x,y)}{(\partial_y J(x,y))^2}-\frac{\partial_y^3 J(x,y)}{(\partial_y J(x,y))^3}+ \frac{3(\partial_y^2 J(x,y))^2}{2(\partial_y J(x,y))^4}
\end{equation}
As $J$ is at most Liouvillian, then $\partial_y J$ is hyperexponential, and as the right hand side is homogeneous in $J$ of degree $-2$, the right hand side is a hyperexponential function we will note $U(x,y)$. Let us consider the equation $U(x,y)=h$. If $U$ is not constant, this defines a curve $\mathcal{C}_h$.

If $X$ does not admit a rational or Darbouxian first integral, then $U$ is not a first integral as it is the exponential of a Darbouxian function. Thus there exists $h$ such that $\mathcal{C}_h$ is not an orbit of $X$, and thus on which $J$ is not constant.  Thus restricting equation \eqref{eq1} on $\mathcal{C}_h$ gives that 
$$\left(\frac{f'''}{f''}-\frac{3}{2}\left(\frac{f''}{f'}\right)^2\right)(z)=h$$
at least on an open set of $z\in\mathbb{C}$.
If $X$ does not admit a rational first integral, but $J$ is Darbouxian then $U$ is algebraic, and thus $\mathcal{C}_h$ is an algebraic curve. As $X$ does not admit a rational first integral, it has at most finitely many Darboux polynomials, and thus there exists $h\in\mathbb{C}$ such that $J_{\mid \mathcal{C}_h}$ is not constant. Thus restricting equation \eqref{eq1} on $\mathcal{C}_h$ gives that 
$$\left(\frac{f'''}{f''}-\frac{3}{2}\left(\frac{f''}{f'}\right)^2\right)(z)=h$$
at least on an open set of $z\in\mathbb{C}$. Remark that if $U$ is constant, this equation also holds.

Solving this differential equation, we find solutions of the form
$$f(z)=a\tan(bz+c)+d,\quad f(z)=\frac{az+b}{cz+d}$$
Thus in both cases, $f(J(x,y))$ is a Liouvillian function. According to Lemma \ref{lem0}, this would be detected in steps $1-5$.

If $X$ admits a rational first integral $J$, then $U$ is rational and is a function of $J$. We can assume that $J$ is indecomposable, and then $U$ is a rational function of $J(x,y)$, and we note $U(x,y)=-2u(J(x,y))$. Thus $f$ satisfies the equation
$$\left(\frac{f'''}{f''}-\frac{3}{2}\left(\frac{f''}{f'}\right)^2\right)(z)=-2u(z)$$
This means that $f=f_1/f_2$ where $f_1,f_2$ form a basis of solutions of the equation
\begin{equation}\label{eq3}
f_i''(z)-u(z)f_i(z)=0
\end{equation}
Let us consider $z_0$ a non algebraic singularity of equation \eqref{eq3}. There exists a formal series expansion of the solutions $f_1,f_2$ near $z_0$ of the form
$$e^{\sum\limits_{i=1}^{n} \frac{a_i}{(z-z_0)^{i/p}}}(z-z_0)^{\alpha} \ln(z-z_0)^{\epsilon} (1+c_1(z-z_0)+c_2(z-z_0)^2+\dots)$$
with $n,p\in\mathbb{N},\alpha\in\mathbb{C},\epsilon\in\{0,1\},a,c\in\mathbb{C}$. Taking the quotient $f_1/f_2$ we can have
$$f(z)=e^{\sum\limits_{i=1}^{n} \frac{a_i}{(z-z_0)^{i/p}}}(z-z_0)^{\alpha} (1+c_1(z-z_0)+c_2(z-z_0)^2+\dots) \hbox{ or}$$
$$f(z)=e^{\sum\limits_{i=1}^{n} \frac{a_i}{(z-z_0)^{i/p}}}(z-z_0)^{\alpha}\ln(z-z_0) (1+c_1(z-z_0)+c_2(z-z_0)^2+\dots),\;\; \alpha\in\mathbb{N}$$
Remark that as the Wronskian of the system is constant, the log case can in fact appear only when $\alpha=0$ and all the $a_i=0$. Moreover, in the first case if the exponential part vanishes ($a_i=0$), then $\alpha$ cannot be rational as else $z_0$ would be an algebraic singularity.
Let us now look at the curve $\mathcal{C}$ given by an irreducible factor of $\hbox{num}(J(x,y)-z_0)$, and not a vertical line. Near $z_0$, $f(z)$ has either an essential (when the $a_i$ are not all zero), an infinitely ramified (when $a_i=0$ and $\alpha \notin \mathbb{Q}$), a logarithmic (when $\ln$ appear) singularity. Now composing with $J(x,y)$ and looking near $\mathcal{C}$, $f(J(x,y))$ will still have a singularity/ramification, and thus $\mathcal{C}$ will be a pole of $F$.

Let us consider $z_0$ an algebraic non meromorphic singularity of equation \eqref{eq3}. Then $f$ admits a series expansion of the form
$$f(z)=(z-z_0)^{\alpha} (1+c_1(z-z_0)+c_2(z-z_0)^2+\dots),\quad \alpha\in\mathbb{Q}/ \mathbb{Z}$$
Let us now look at the curve $\mathcal{C}$ given by an irreducible factor of $\hbox{numer}(J(x,y)-z_0)$, and note $m$ the multiplicity of this factor. The function $f(J(x,y))$ will have a singularity near $\mathcal{C}$ if $m$ is not a multiple of the denominator of $\alpha$.

Thus if equation \eqref{eq3} has a least $2$ non algebraic singularity, we can put them by Moebius transformation at $0,\infty$ and then the factors of the numerator and denominator of $J$ (except possibly those in $x$ only) will appear as poles of $F$. Thus $\partial_y \ln J$ has poles of order $1$ which are poles of $F$ (the factors depending on $x$ only then disappearing due to the $y$ derivation). This is a Darbouxian first integral whose defining function has a denominator dividing $\hbox{sf}(\hbox{den}(F))=Q$. As $\ln J$ is a linear combination of logs, its $y$ derivative can then be written $P/Q$ with $\deg P \leq \deg Q-1$. So step $6$ would detect the equation defining the Darbouxian first integral $\ln J$. 

If equation \eqref{eq3} has only one non algebraic singularity, then we can put it at infinity by a Moebius transformation. Let us first assume the singularity at infinity to be irregular. Then $u(z)$ in equation \eqref{eq3} has an asymptotic behaviour at infinity $u(z)=a z^r+o(z^r)$ with $r\geq -1$. We have the relation
\begin{equation}\label{eqirr}
-2F(x,y)=(\partial_y J(x,y))^2 (-2u(J(x,y))+
\left(\frac{\partial_y^3 J(x,y)}{\partial_y J(x,y)}-\frac{3}{2} \left( \frac{\partial_y^2J(x,y)}{\partial_yJ(x,y)}\right)^2\right).
\end{equation}
Consider an irreducible factor of the denominator of $J$ (not in $x$ only), and note its multiplicity $p$. The second part of the right hand side will have a pole of order at most $2$, and the first part has a pole of order $pr+2p+2$. As $r\geq -1$, this order is always strictly more than $2$, and thus the poles cannot compensate. Thus the multiplicity of the factor in the denominator of $J$ is less than the multiplicity of the factor in the denominator of $F$, and thus the denominator of $J$ divides the denominator of $F$.

There are possibly also factors in the denominator of $J$ depending on $x$ only. The second part of the right hand side will have no pole along such vertical straight line, and for the first part, noting $p$ the multiplicity, we have a pole of multiplicity $2p+pr=p(r+2)$. As $r\geq -1$, the multiplicity of such factor in the denominator of $J$ is less than the multiplicity of the factor in the denominator of $F$, and thus the denominator of $J$ divides the denominator of $F$.

We now look at the line at infinity. The second part of the right hand side is of degree at most $-2$. Noting $p=\deg \hbox{num}(J)-\deg\hbox{den}(J)$, the first term has a pole of order $2p-2+rp=p(r+2)-2$. Thus if $p\geq 1$, there cannot be compensation, and thus $p(r+2)-2=\deg \hbox{num}(F)-\deg \hbox{den}(F)$. Then $p\leq \max(1,\deg \hbox{num}(F)-\deg \hbox{den}(F)+2)$. Thus the rational first integral $J$ has a numerator and denominator degree $\leq \max(\deg \hbox{num}(F)+2,\deg \hbox{den}(F)+1)$. This rational first integral would be detected in Step $7$.\\

We now assume that the singularity at infinity is regular. As it is non algebraic, it can be an irrational exponent difference or a log. In both cases, the factors of the denominator of $J$ will be factors of the denominator of $F$, or if it is a factor in $x$ only, a factor of numerator of $A$ (these are the new singularities possibly coming from the second equation of \eqref{eq2b}). So noting $Q=\hbox{sf}(\hbox{den}(F)A)$, we can write $J=H/Q^n$ for some $n\in\mathbb{N}, H\in\mathbb{Q}[x,y]$. Let us note $z_1,\dots,z_p$ the algebraic non meromorphic singularities, and $m_i$ their ramification index. Let us note $k$ the lcm of the $m_i$. The factors of $\hbox{num}(J-z_i)$ are either with multiplicity multiple of $m_i$ or have to be factors of $Q$. Thus the factors of $\hbox{num}(J-z_i)^{k/m_i}$ are either with multiplicity multiple of $k$ or have to be factors of $Q$. Making a product over the $i$, we obtain
$$\prod_{i=1}^p \hbox{num}(J-z_i)^{k/m_i} =V^k M$$
where $V\in\mathbb{Q}(x,y)$ and $M$ is a product of factors of $Q$ (with unknown multiplicity). Now diving this relation by a suitable power of $\hbox{den}(J)$ we have
$$\prod_{i=1}^p (J-z_i)^{k/m_i} =V^k M$$
with still $V\in\mathbb{Q}(x,y)$ and $M$ is still a product of factors of $Q$ (possibly with negative exponents). The rational function $M$ can be assumed to be polynomial and to have factors with multiplicities $\leq k-1$ by changing $V$.

\begin{lem}\label{lem1}
Let be $P\in\mathbb{Q}[z]$ a polynomial with roots of multiplicity $<k\in \mathbb{N}/\{0,1\}$, $S\in\mathbb{Q}[x,y]$ a square free polynomial and $M\in\mathbb{Q}[x,y]$ dividing a power of $S$. Consider the equation $P(J)=V^k M$ with unknowns $J,V\in \mathbb{Q}[x,y,S^{-1}]$. When either
\begin{itemize}
\item $k\geq 3 \hbox{ odd},\sharp P^{-1}(0)\geq 2$
\item $k\geq 3 \hbox{ even},\sharp P^{-1}(0)\geq 2, P$ is not a $k/2$-th power of a polynomial
\item  $k=2, \sharp P^{-1}(0)\geq 3$
\end{itemize}
then we have
$$\max(\deg \hbox{num}(J),\deg \hbox{den}(J)) \leq 6\max(0,\deg S-1) $$
\end{lem}

\begin{proof}[Proof of Lemma \ref{lem1}]
We first consider the case $k\geq 3$. From now we will consider $x$ as a variable and $y$ as a parameter. Let us note $z_1,z_2$ two distinct roots of $P$ with multiplicity $m_1,m_2$ and such that $m_1\wedge k \neq k/2$ when $k$ even. We can always find such two points as $\sharp P^{-1}(0)\geq 2$ and when $k$ is even, $P$ is not a $k/2$-th power of a polynomial, and thus all roots are not of multiplicity $k/2$.
 and write the equality
$$(J-z_1)+(-J+z_2)+(z_1-z_2)=0$$
The three terms on this sum have the same valuations (as function in $x$) except at the zeros of $J-z_1$, $J-z_2$, $S$ and the point infinity. Let us note $v$ the number of such points. We will now apply the fundamental inequality of Mason \cite{mason1984diophantine} which gives (the field is $\mathbb{C}(y)(x)$ and thus of genus $0$)
\begin{equation}\label{eq4}
H\left( \frac{J-z_1}{z_1-z_2}\right) \leq \max(0,v-2)
\end{equation}
where the height of a rational function is given by $H(f)=\sum_{z\in\mathbb{P}^1} \min(0,\hbox{val}_z(f))$.\\
Factoring the equation $P(J)=V^k M$, we have
$$(J-z_1)^{m_1}(J-z_2)^{m_2} \prod\limits_{i=3}^{\sharp P^{-1}(0)} (J-z_i)^{m_i}=MV^k$$
Two different levels of $J$ cannot have a common level, and thus  $J-z_i$ admits factors either of $S$ (recall that $M$ has factors of $S$ also), either they are multiple of $k/(m_i \wedge k)$. As $m_1\wedge k \neq k/2$, we have $k/(m_1 \wedge k)\geq 3$, and the number of distinct roots of $J-z_1$ outside $S^{-1}(0)$ is at most $\frac{1}{3}\max(\deg \hbox{num}(J),\deg \hbox{den}(J))$. Similarly, the number of distinct roots of $J-z_2$ outside $S^{-1}(0)$ is at most $\frac{1}{2}\max(\deg \hbox{num}(J),\deg \hbox{den}(J))$. Thus equation \eqref{eq4} becomes
\begin{equation}\label{eq5}
H\left( \frac{J-z_1}{z_1-z_2}\right) \leq \frac{5}{6}\max(\deg \hbox{num}(J),\deg \hbox{den}(J))+ \max(0,\sharp S^{-1}(0)-1)
\end{equation}
Now the height of $J-z_1$ is $\deg \hbox{den}(J)$ when $\deg \hbox{den}(J) \leq \deg \hbox{num}(J-z_1)$, and when $\deg \hbox{num}(J-z_1)> \deg \hbox{den}(J)$, the function $J-z_1$ has a pole at infinity of order $\deg \hbox{num}(J-z_1)-\deg \hbox{den}(J)$. Thus the height of $J-z_1$ is $\deg \hbox{num}(J-z_1)$. Now except when $J$ is constant, we have $\deg \hbox{num}(J-z_1)=\deg \hbox{num}(J)$. Thus the inequality \eqref{eq5} becomes
\begin{equation}\begin{split}\label{eq6}
\max(\deg \hbox{num}(J),\deg \hbox{den}(J)) & \leq   \frac{5}{6}\max(\deg \hbox{num}(J),\deg \hbox{den}(J))+ \max(0,\sharp S^{-1}(0)-1)\\
\frac{1}{6}\max(\deg \hbox{num}(J),\deg \hbox{den}(J)) & \leq  \max(0,\sharp S^{-1}(0)-1)\\
\end{split}\end{equation}
Now as $S$ has simple roots for a generic value of the parameter $y$, we have $\sharp S^{-1}(0)= \deg_x S$. Doing the same reasoning as before exchanging $x,y$, we obtain the bound of the Lemma.

We now consider the case $k=2$. We now follow \cite{mason1984diophantine}, and more specifically Proposition 8.2 of \cite{hindry1988canonical}. As before, we will consider $x$ as a variable and $y$ as a parameter. All the roots of $P$ are simple (as their multiplicity should be $<2$), and as $\deg P \geq 3$, it is always possible to choose $z_1,z_2,z_3$ three distinct roots of $P$. We now consider the field extension $\mathbb{K}=\mathbb{C}(y)(x,\sqrt{J-z_1},\sqrt{J-z_2},\sqrt{J-z_3})$. Let us note $v'$ the number of places of $\mathbb{K}$ lying above the zeros of $S$ and the point infinity. Writing the equality
$$(\sqrt{J-z_1}-\sqrt{J-z_2})+(\sqrt{J-z_2}-\sqrt{J-z_3})+(\sqrt{J-z_3}-\sqrt{J-z_1})=0$$
we apply the fundamental inequality of Mason
$$H_\mathbb{K}\left( \frac{\sqrt{J-z_1}-\sqrt{J-z_2}}{\sqrt{J-z_2}-\sqrt{J-z_3}}\right) \leq \max(0,2g-2+v')$$
where $g$ is the genus of $K$ and
\begin{equation}\label{eq7}
H_\mathbb{K}(f)=\sum_{z \hbox{ places of } \mathbb{K}} \nu_z \min(0,\hbox{val}_z(f))
\end{equation}
where $\nu_z$ is the ramification index of $\mathbb{K}$ at $z$. Remark that if $f\in\mathbb{C}(y)(x)$, then $H_\mathbb{K}(f)=[\mathbb{K}:\mathbb{C}(y)(x)] H(f)$. Also inequality \eqref{eq7} is valid for all Galois conjugates of the square roots, giving us in fact four different inequalities. We deduce
$$H_\mathbb{K}\left( \frac{\sqrt{J-z_1}-\sqrt{J-z_2}}{\sqrt{J-z_2}\pm\sqrt{J-z_3}}+\frac{-\sqrt{J-z_1}-\sqrt{J-z_2}}{\sqrt{J-z_2}\pm \sqrt{J-z_3}}\right) \leq 2\max(0,2g-2+v')$$
$$H_\mathbb{K}\left(\frac{\sqrt{J-z_2}}{\sqrt{J-z_2}\pm\sqrt{J-z_3}}\right) \leq 2\max(0,2g-2+v')$$
$$H_\mathbb{K}\left(\frac{\sqrt{J-z_2}}{\sqrt{J-z_2}+\sqrt{J-z_3}} \frac{\sqrt{J-z_2}}{\sqrt{J-z_2}-\sqrt{J-z_3}}\right) \leq 4\max(0,2g-2+v')$$
$$H_\mathbb{K}\left(\frac{J-z_2}{z_3-z_2} \right) \leq 4\max(0,2g-2+v')$$
Now formula from \cite{hindry1988canonical} gives $2g-2+v'=[\mathbb{K}:\mathbb{C}(y)(x)](\sharp S^{-1}(0)-1)$ and thus
$$H\left(\frac{J-z_2}{z_3-z_2} \right) \leq 4\max(0,\sharp S^{-1}(0)-1)$$
Previous reasoning gives that the left hand side is $\max(\deg \hbox{num}(J),\deg \hbox{den}(J))$ and $\sharp S^{-1}(0)=\deg_x S$. Doing the same exchanging $x,y$ gives the bound $\max(\deg \hbox{num}(J),\deg \hbox{den}(J)) \leq 4\max(0,\deg S-1) $ which implies the one of the Lemma.
\end{proof}

As infinity is a regular singularity with non rational exponent difference or log, there cannot be only one algebraic singularity as then the monodromy around infinity and this finite singularity should be the same, which is impossible as one is of finite order and the other infinite order. Thus there are at least $2$ finite singularities, thus $\sharp P^{-1}(0)\geq 2$. Let us first look at the case $k\geq 3$. When $k$ is even, we cannot have $P$ a $k/2$th power as then $k/m_i=k/2,\;i=1\dots p$ and thus $\hbox{lcm}((m_i)_{i=1\dots p})=2 \neq k$. Thus Lemma \ref{lem1} applies and gives
$$\max(\deg \hbox{num}(J),\deg \hbox{den}(J)) \leq 6\max(0,v-1) $$
where $v$ is the degree of $Q$. Step $7$ would detect this case.\\
We now look at the case $k=2$. If $\sharp P^{-1}(0)\geq 3$, Lemma \ref{lem1} applies and gives the same inequality. Step $7$ would detect this case. If $\sharp P^{-1}(0)=2$, we can put the two finite singularities at $-1,1$. Now the possible equations \eqref{eq3} have a basis of solutions of the form
$$(1-z^2)^{1/4} e^{\frac{1}{2}i\lambda\hbox{arcsin}(z)}, (1-z^2)^{1/4} e^{-\frac{1}{2}i\lambda\hbox{arcsin}(z)} \lambda\in\mathbb{C}^* $$
$$(1-z^2)^{1/4} , (1-z^2)^{1/4} \hbox{arcsin}(z) $$
Thus $\mathcal{F}$ would be Liouvillian, and according to Lemma \ref{lem0}, this would have been detected in steps $1-5$.

We finally consider the last case when all singularities are regular algebraic. Let us note $z_1,\dots,z_p$ the singularities, and $m_i$ their ramification index. Let us note $k$ the lcm of the $m_i$. Up to Moebius transformation, we can assume that equation \eqref{eq3} has no singularities at infinity. The factors of $\hbox{num}(J-z_i)$ are either with multiplicity multiple of $m_i$ or have to be factors of $Q$. Thus the factors of $\hbox{num}(J-z_i)^{k/m_i}$ are either with multiplicity multiple of $k$ or have to be factors of $Q$. Making a product over the $i$, we obtain
$$\prod_{i=1}^p \hbox{num}(J-z_i)^{k/m_i} =V^k M$$
where $V\in\mathbb{Q}(x,y)$ and $M$ is a product and quotient of factors of $Q$. Now diving this relation by a suitable power of $\hbox{den}(J)$ we have
\begin{equation}\label{eq8}
\prod_{i=1}^p (J-z_i)^{k/m_i} =V^k M
\end{equation}
with still $V\in\mathbb{Q}(x,y)$ and $M$ is a product and quotient of factors of $Q$. By changing $V$, we can assume $M$ to be polynomial and with multiplicities $\leq k-1$.

\begin{lem}\label{lem2}
The curves $y^k=P(x)$ which are of genus $\geq 2$ are of genus $\geq k/12+1$.
\end{lem}

\begin{proof}[Proof of Lemma \ref{lem2}]
We first recall the formula for the genus of a superelliptic curve \cite{shaska2015case}
$$g=\frac{1}{2} \left(k(\sharp B-2)- \sum\limits_{z\in B} k \wedge m_z\right)+1 $$
where $B$ is the set of ramification points of the curve on $\mathbb{P}^1$ and $m_z$ the multiplicity of $z$ as a root of the defining polynomial $P$.

Let us first remark that when $\sharp B=2$, the genus is $0$ and thus excluded. Thus $\sharp B\geq 3$.

As $k \wedge m_z$ divides $k$, it can be equal to $k/2,k/3,k/4,\dots$. Let us remark that we have the relation on multiplicities $\sum\limits_{z\in B} m_z=0 \hbox{ mod } k$.
Thus for any given $k$, we want to maximize the quantity
$$M_k=\max \left( \sum\limits_{z\in B} k \wedge m_z \right) \hbox{ under constraints}$$
$$\sum\limits_{z\in B} m_z=0 \hbox{ mod } k,\;\; \hbox{gcd}((m_z)_{z\in B} )=1,\;\; m_z\in\{1,\dots,k-1\}$$
The gcd condition is necessary to ensure our curve is irreducible (else by convention $g=-1$), and we can always assume $m_z\in\{1,\dots,k-1\}$ by a variable change on $y$.\\

We have $k \wedge m_z\leq k/2$ and thus  $M_k\leq k/2 \sharp B$ and thus $g\geq k/4 (\sharp B -4)+1$. Thus when $\sharp B \geq 5$, the inequality of the Lemma is trivially satisfied. Remark now that if all $m_z$ are multiple of $k/2$ we have $\hbox{gcd}((m_z)_{z\in B} ) = k/2$ which is forbidden except if $k=2$. This gives a better bound for $M_k\leq k/2 \sharp B-k/6$, and thus
$$g\geq \frac{1}{2}(k/2 \sharp B-2k+k/6)+1=\frac{k}{12} (3\sharp B-11)+1,\;\; k\geq 3$$
Thus for $\sharp B=4, k\geq 3$, this inequality suffices to prove the Lemma. For $k=2$, this is an elliptic curve, and thus genus is $1$, and thus excluded.\\

The only case left is $\sharp B=3$, and we can put by Moebius transformation the ramification points at $0,1,\infty$. The triplet of values of $k \wedge m_z$ which lead to $M_k>5k/6$ are the following
$$(k/2,k/2,k/j)_{j\geq 2},\; (k/2,k/3,k/j)_{j\geq 3},$$
$$(k/2,k/4,k/j)_{j=4\dots 11},\; (k/2,k/5,k/j)_{j=5,6,7},\; (k/3,k/3,k/j)_{j=3\dots 5}$$ 
As before, due to the relation $\sum\limits_{z\in B} m_z=0 \hbox{ mod } k$, the first case implies that $j=2$ and then $k=2$. This leads to an elliptic curve, thus excluded. In the second case, we need that $k/j$ to be a linear combination (with integer coefficients) of $k/2,k/3$, and thus a multiple of $k/6$. Thus the possible $j$ are $3,6$. Then the only possibility is $k=6$, and the possible curves equations are (only the triplet $k/2,k/3,k/6$ is possible)
$$y^6=x^3(x-1)^2,\quad y^6=x^3(x-1)^4.$$
These curves are of genus $1$, and thus excluded.
For the third case, the possible equations would be
$$y^k=x^{k/2} (x-1)^{k/4},\quad y^k=x^{k/2} (x-1)^{3k/4}$$
and thus ramification at infinity is $k/4,3k/4$, and thus $j=4$ and so $k=4$. The curves are then $y^4=x^2(x-1)$, $y^4=x^2(x-1)^3$ which are of genus $1$, thus excluded.
In the fourth case, we need that $k/j$ to be a linear combination (with integer coefficients) of $k/2,k/5$, and thus a multiple of $k/10$. Thus the only possibility is $j=5$, and thus $k=10$. However the possible $m_z$ can be $(5,2r,2l)$ and the sum condition $5+2r+2l=0 \hbox{ mod }10$ cannot be satisfied. Rests the fifth case, for which only $j=3$ is possible (as $k/j$ should be a multiple of $k/3$). In this case we have $k=3$ and the possible curves are
$$y^3=x(x-1),\;y^3=x^2(x-1),\; y^3=x^2(x-1)^2.$$
All these curves are of genus $1$. Thus $M_k\leq 5k/6$ in all relevant cases, and so
$$g\geq \frac{1}{2}(k-5/6k)+1=\frac{k}{12}+1$$

\end{proof}

Let us consider equation \eqref{eq8} with $y$ seen as a parameter and unknowns $J,V\in\mathbb{C}(y)(x)$. Such a rational solution defines a rational morphism from the superelliptic curve $M(X,y)=Y^k$ of genus $g'$ to the superelliptic curve $\prod_{i=1}^p (X-z_i)^{k/m_i}=Y^k$ of genus $g$ by
$$\varphi(X,Y)=(J(X,y),V(X,y)Y).$$
Using the Riemann Hurwitz formula \cite{oort2016riemann}, we have that
$$2g'-2\geq \deg \varphi  (2g-2)$$
where the degree of $\varphi$ is the generic number of points of a fibre of $\varphi$. Looking at the expression of $\varphi$, we have $\deg \varphi= \max(\deg \hbox{num}(J),\deg \hbox{den}(J))$, and thus when $g\geq 2$
$$\max(\deg \hbox{num}(J),\deg \hbox{den}(J)) \leq \frac{g'-1}{g-1}.$$
As $M$ is formed from factors of $Q$, its number of roots (seen as a function of $x$) is less than $\deg_x Q$. Now using the formula for the genus of a superelliptic curve, we get
$$g'\leq \frac{1}{2} ((k-1)(\deg_x Q+1)-2k)+1$$
and with Lemma \ref{lem2} with $g\geq 2$ we have 
\begin{equation}\label{eq9}
\max(\deg \hbox{num}(J),\deg \hbox{den}(J)) \leq \frac{(k-1)(\deg_x Q+1)-2k}{k/6}\leq 6(\deg_x Q-1)
\end{equation}

Exchanging the role of $x,y$, we obtain that $\max(\deg \hbox{num}(J),\deg \hbox{den}(J))\leq 6(\deg Q-1)$. Such rational first integral would be detected in step $7$. Now stays the case when $g=0,1$. We will now prove that these remaining cases never occur in the algorithm.

\begin{lem}\label{lem3}
The curves $y^k=P(x)$ which are of genus $0$ are up to Moebius transformation of $x$ of the form $y^k=x^l$, and those of genus $1$ are of the form
$$y^2=x(x-1)(x-a),\; y^3=x(x-1),\; y^3=x^2(x-1)^2,$$
$$y^4=x^2(x-1),\; y^4=x^2(x-1)^3,\; y^6=x^3(x-1)^2,\; y^6=x^3(x-1)^4$$
\end{lem}

\begin{proof}[Proof of Lemma \ref{lem3}]
We want $g=0$ or $g=1$, which means
\begin{equation}\label{eq10}
k(\sharp B-2)\leq  \sum\limits_{z\in B} k \wedge m_z.
\end{equation}
As $k \wedge m_z\leq k/2$, the right hand side is less than $k/2 \sharp B$, and thus $\sharp B-2\leq 1/2 \sharp B$ and so $\sharp B \leq 4$. If $\sharp B=4$, equation \eqref{eq10} becomes
$$2k\leq \sum\limits_{z\in B} k \wedge m_z$$
This requires that $k \wedge m_z=k/2$ for all $m_z$, and thus $k=2$. This defines an elliptic curve, and putting three of the $4$ singularities at $0,1,\infty$ by Moebius transformation gives the first equation.
Now for $\sharp B=3$, we put singularities at $0,1,\infty$ by Moebius transformation, and equation \eqref{eq10} becomes
$$k\leq \sum\limits_{z\in B} k \wedge m_z.$$
This gives the possible triples for $k \wedge m_z$
$$(k/2,k/2,k/j)_{j\geq 2},\; (k/2,k/3,k/j)_{j=3 \dots 6}, (k/2,k/4,k/4), (k/3,k/3,k/3)$$
The first case requires $k/j$ multiple of $k/2$, and so $j=2$, and then $k=2$. This means all $m_z=2$ and the relation $\sum\limits_{z\in B} m_z=0 \hbox{ mod } 2$ is not satisfied.\\
In the second case, we have $k/j$ multiple of $k/6$, and so $j=3,6$, and then $k=6$. The case $j=3$ is not realizable as $3+2r+2l=0  \hbox{ mod } 6$ has no solutions. Then stays $j=6$ which gives
$$y^6=x^3(x-1)^2,\; y^6=x^3(x-1)^4.$$
The second case leads to $k=4$, and the curves
$$y^4=x^2(x-1),\; y^4=x^2(x-1)^3.$$
The last case gives $k=3$, and the possible curves are
$$y^3=x(x-1),\; y^3=x^2(x-1)^2$$
\end{proof}

Each case of Lemma \ref{lem3} corresponds to a specific set of singularities with exponent difference in \eqref{eq3}. We will prove that in all these cases, the Galois group of \eqref{eq3} is in fact virtually solvable, and thus these cases would be detected in steps $1-5$. In the case of genus $0$, equation \eqref{eq3} has only two regular singularities at $0,\infty$, thus monodromy group is single generated and thus Abelian. As equation is Fuchsian, so is the Galois group, and thus is solvable. We now look at genus $1$ cases. In the first case, there are $4$ regular singularities in equation \eqref{eq3} with exponent difference in $1/2+\mathbb{Z}$. Up to multiplication by a hyperexponential function (which does not impact the quotient of two independent solutions), equation \eqref{eq3} becomes a Heun equation \cite{maier2007192}
\begin{equation}\begin{split}\label{eq11}
f''_i(z)-\left( \frac{n_1-1/2}{z}+\frac{n_2-1/2}{z-1}+\frac{n_3-1/2}{z-a} \right) f_i'(z)\\
-\frac{(n_4-n_3-n_1-n_2+1)(n_4+n_3+n_1+n_2)z+4q}{4(z-a)z(z-1)}f_i(z)=0
\end{split}\end{equation}
where $n_1,n_2,n_3,n_4\in\mathbb{Z}, a\in\mathbb{C}-\{0,1\}, q\in\mathbb{C}$.

\begin{lem}\label{lem4}
Equation \eqref{eq11} has dihedral Galois group.
\end{lem}

\begin{proof}[Proof of Lemma \ref{lem4}]
The monodromy matrix around zero is given by matrix
$$A=\left(\begin{array}{cc} 1&0\\ 0&-1\\\end{array}\right)$$
Now in the same basis, we define $B$ the mondromy matrix around $1$. We know that $B$ has eigenvalues $-1,1$ but as we cannot choose the basis (it is already chosen to diagonalize $A$), we can only write
$$B=\left(\begin{array}{cc} 2p_2p_3+1&2(p_2p_3+1)p_2\\ -2p_3&-2p_2p_3-1\\\end{array}\right)$$
with $p_2,p_3\in\mathbb{C}$ unknown, as any matrix with eigenvalues $-1,1$ can be written under this form. We have the same from the monodromy matrix $C$ around $a$, which can then be written
$$C=\left(\begin{array}{cc} 2q_2q_3+1&2(q_2q_3+1)q_2\\ -2q_3&-2q_2q_3-1\\\end{array}\right)$$
We still have that the monodromy matrix around infinity should have eigenvalues $-1,1$, which means that $ABC$ should have eigenvalues $-1,1$. This gives the constraint
$$p_2^2+\frac{p_2}{p_3}=q_2^2+\frac{q_2}{q_3}$$
Introducing $s$ this quantity, we can solve in $p_3,q_3$ giving $p_3=p_2/(s-p_2^2)$ and $q_3=q_2/(s-q_2^2)$. Substituting this in the expression of $B,C$ and dropping the index $2$ of the $p_2,q_2$ gives
$$B=\left(\begin{array}{cc} -\frac{p^2+s}{p^2-s}&  -\frac{2sp}{p^2-s}\\ \frac{2p}{p^2-s} & \frac{p^2+s}{p^2-s}\\\end{array}\right),\;\; 
C=\left(\begin{array}{cc} -\frac{q^2+s}{q^2-s}&  -\frac{2sq}{q^2-s}\\ \frac{2q}{q^2-s} & \frac{q^2+s}{q^2-s}\\\end{array}\right)$$
We will now prove that the group $\langle A,B,C \rangle$ is a dihedral group. We have that $A^2=B^2=C^2=I_2$ the identity matrix, thus any word can be written without repeated letter. We then consider the sub group of index $2$ formed by words of even length
$$G=\langle AB,AC,BA,BC,CA,CB \rangle$$
Computing the commutator of every pair of these generators, we always find $0$, and thus $G$ is Abelian. Thus $\langle A,B,C \rangle$ is a dihedral group, and so equation \eqref{eq11} has a dihedral monodromy group, and so dihedral Galois group.
\end{proof}

Using Lemma \ref{lem4}, Galois group of equation \eqref{eq11} is dihedral and thus virtually solvable. Equation \eqref{eq8} requires that the multiplicities of the factors of $P$ divide $k$, thus only the second, fourth and sixth cases are possible.\\
In the second case, we have three singularities with exponent difference in $1/3+\mathbb{Z}$ giving for equation \eqref{eq3}
$$f''_i(z)-\frac{1}{36}\left(
\frac{9k_1^2-12k_1-5}{z^2}+\frac{9k_1^2+9k_2^2-9k_3^2-12k_1+6k_2-6k_3-5}{z}+\right.$$
$$\left. \frac{9k_2^2+6k_2-8}{(z-1)^2}+\frac{-9k_1^2-9k_2^2+9k_3^2+12k_1-6k_2+6k_3+5}{z-1}
\right)f_i(z)=0$$
with $k_1,k_2,k_3\in\mathbb{Z}$. This equation is a hypergeometric equation (in $Q$ form). According to Kimura table \cite{kimura1969riemann}, this equation always has a solvable Galois group.

In the fifth case, we have three singularities with exponent difference $1/2+\mathbb{Z},1/4+\mathbb{Z},1/4+\mathbb{Z}$ giving for equation \eqref{eq3}
$$f''_i(z)+\frac{1}{64}\left(
4{\frac {4\,{k_{{1}}}^{2}-4\,k_{{1}}-3}{{z}^{2}}}+4\,{\frac {4\,{k_{
{1}}}^{2}+4\,{k_{{2}}}^{2}-4\,{k_{{3}}}^{2}-4\,k_{{1}}+2\,k_{{2}}-2\,k
_{{3}}-3}{z}}+\right.$$
$$\left. 4\,{\frac {-4\,{k_{{1}}}^{2}-4\,{k_{{2}}}^{2}+4\,{k_{{3}
}}^{2}+4\,k_{{1}}-2\,k_{{2}}+2\,k_{{3}}+3}{z-1}}+{\frac {16\,{k_{{2}}}
^{2}+8\,k_{{2}}-15}{ \left( z-1 \right) ^{2}}}
\right)f_i(z)=0$$
with $k_1,k_2,k_3\in\mathbb{Z}$. This equation is a hypergeometric equation (in $Q$ form). According to Kimura table \cite{kimura1969riemann}, this equation always has a solvable Galois group.

In the seventh case, we have three singularities with exponent difference $1/2+\mathbb{Z},1/3+\mathbb{Z},1/6+\mathbb{Z}$ giving for equation \eqref{eq3}
$$f''_i(z)+\frac{1}{144}\left(
9\,{\frac {4\,{k_{{1}}}^{2}-4\,k_{{1}}-3}{{z}^{2}}}+12\,{\frac {3\,{k_
{{1}}}^{2}+3\,{k_{{2}}}^{2}-3\,{k_{{3}}}^{2}-3\,k_{{1}}+2\,k_{{2}}-k_{
{3}}-2}{z}}+\right.$$
$$\left.12\,{\frac {-3\,{k_{{1}}}^{2}-3\,{k_{{2}}}^{2}+3\,{k_{{3}}
}^{2}+3\,k_{{1}}-2\,k_{{2}}+k_{{3}}+2}{z-1}}+4\,{\frac {9\,{k_{{2}}}^{
2}+6\,k_{{2}}-8}{ \left( z-1 \right) ^{2}}}
\right)f_i(z)=0$$
with $k_1,k_2,k_3\in\mathbb{Z}$. This equation is a hypergeometric equation (in $Q$ form). According to Kimura table \cite{kimura1969riemann}, this equation always has a solvable Galois group.
This ends the possibilities with equation \eqref{eq3} having only algebraic singularities, and thus if none of the previous cases happened, the Riccati first integral does not simplify.
\end{proof}

Remark that in the final step, we use for $Q=\hbox{sf}(\hbox{den}(F)A)$. Still in the proof, the polynomial $A$ is present here only because of possible factors in $x$ only. At first it seems profoundly inefficient to consider in $Q$ all the factors of $A$, including those not in $x$ only. However, factors in $x,y$ in $A$ will also typically appear in $F$ as they are singular curves for the derivative $\partial_x y(x)$. Thus these additional factors will in fact simplify when taking the square free part.\\

We want now to ensure, that the returned output always has rational coefficients. Indeed, the output of this algorithm will quite often be used to further reduce the first integral, and this is a requirement for running \underline{\sf ReduceLiouvillian}, \underline{\sf ReduceDarbouxian}.

\begin{prop}
The coefficients of the returned equations are always in $\mathbb{Q}$.
\end{prop}

\begin{proof}
In step $1$, rational solutions are computed, and a basis of such solutions with coefficients in $\mathbb{Q}$ always exists \cite{barkatou2012computing}. Thus the returned Darbouxian first integral has rational coefficients. In step $2$, the rational solution again is computed with rational coefficients. Thus the special polynomial also has rational coefficients. If it has a double root, then its solution still has rational coefficients. Else we compute the square root of the discriminant $\Delta$. If it is not a square in $\bar{\mathbb{Q}}(x,y)$, then this gives a $2$-Darbouxian first integral with rational coefficients. If it is a square in $\bar{\mathbb{Q}}(x,y)$, after normalization by dividing by the dominant coefficient, it is a square in $\mathbb{Q}(x,y)$, and thus this gives a Darbouxian first integral with rational coefficients.

In step $3$, hyperexponential solutions $R$ are computed. As the fully reducible case has already been detected in steps $1,2$, we know that at most one hyperexponential solution exists. If its logarithmic derivative had coefficients in an extension of $\mathbb{Q}$, the Galois action would define other linearly independent conjugated hyperexponential solutions. Thus at least two would exist, which is impossible. Thus $\partial_y R/R \in\mathbb{Q}(x,y)$, and thus the returned Liouvillian first integral equation has coefficients in $\mathbb{Q}$. In step $4$, the special polynomial will have rational coefficients, and as there are no hyperexponential solutions, the first integral returned is $2$-Darbouxian, and as $\Delta$ has rational coefficients, this equation has rational coefficients.

In step $5$, the rational solutions of the $m$th symmetric powers have rational coefficients, thus so $a_0,a_1$ have rational coefficients, and thus the returned rational first integral.
\end{proof}

Remark that the case of the Galois group $D_2$ is treated in step $5$ and not step $4$. It appears that even if a cubic extension can appear in the coefficient of the minimal polynomial of such solution \cite{hendriks1995galois}, these are removed when considering the quotient of two solutions and by decomposing the algebraic function $\mathcal{H}_1/\mathcal{H}_2$. The returned rational expression can be decomposable, but the decomposition of a rational first integral with coefficients in $\mathbb{Q}$ will also give a first integral with coefficients in $\mathbb{Q}$, thus an algebraic extension of the coefficient field is never necessary.

\section{Reduction of $k$-Darbouxian first integrals}

A $k$-Darbouxian first integral $\mathcal{F}$ can be written under the form
$$\mathcal{F}(x,y)= \int -\frac{B}{A} F dx+ F dy,\quad F^k\in\mathbb{Q}(x,y)$$
Let us note
$$\omega=-\frac{B}{A} F dx+ F dy=\frac{P_1dx+P_2dy}{QS^{1/k}}$$
where $S$ is a polynomial with root multiplicities $<k$ and $P_1,P_2,Q$ polynomials with $Q$ coprime with $\hbox{gcd}(P_1,P_2)$.

\begin{defi}
A closed differential form $\omega$ is said to be reduced if $Q$ has simple poles, is coprime with $S$ and $\max(\deg P_1,\deg P_2) \leq \deg Q +\tfrac{1}{k} \deg S-1$.
A Hermite reduction of $\omega$ is a function $G$ with $G\in S^{-1/k}\mathbb{C}(x,y)$ such that $\omega-d(G)$ is reduced.
\end{defi}

For $k\geq 2$, Hermite reduction are not always possible, as in the case of elliptic integrals of the second kind \cite{combot2021elementary}. However, we have that if $\int \omega$ is elementary, then Hermite reduction is possible, and we can write
$$\int \omega= \frac{G}{S^{1/k}}+\sum \lambda_i \ln F_i,\quad F_i \in\mathbb{C}(x,y,S^{1/k}).$$

\noindent\underline{\sf HermiteReduction}\\
\textsf{Input:} A closed differential form $\omega=(P_1dx+P_2dy)/(QS^{1/k})$\\
\textsf{Output:} A Hermite reduction of $\omega$, or ``None''
\begin{enumerate}
\item Note $\hat{Q}=Q/\hbox{sf}(Q)$ and $G=U/(\hat{Q}S^{1/k})$ with undetermined polynomial $U$ with
$$ \deg U \leq \max(\max(\deg P_1,\deg P_2)-\deg Q+\deg \hat{Q}+1, \deg \hat{Q}+k^{-1}\deg S)$$
\item Compute
$$Eq_1= (P_1dx+P_2dy-S^{1/k} d(G))QS \hbox{ mod } (S\hbox{gcd}(Q,S\partial_x Q,S\partial_y Q))$$
$$Eq_2= (P_1dx+P_2dy-S^{1/k} d(G))QS \hbox{ mod } (x^{d-i}y^i)_{i=0\dots d}$$
where $d=\lfloor\deg Q+(1+\tfrac{1}{k}) \deg S\rfloor$.
\item Solve $Eq_1=Eq_2=0$ in the undetermined coefficients of $U$. If a solution $U$ is found, return $U/S^{1/k}$ else return ``None''.
\end{enumerate}

\begin{prop}
Algorithm \underline{\sf HermiteReduction} finds a Hermite reduction of $\omega$ if and only if one exists.
\end{prop}

\begin{proof}
If a Hermite reduction has a pole of order $p$, then $\omega$ should have this pole of order $p+1$. The candidate $G$ should have for poles the poles of $\omega$ with multiplicity at most one less than the poles of $\omega$, and thus its denominator should divide $\hat{Q}$, which is done in step $1$. The multiplicity of the line at infinity of $G$ if it is a pole can be at most one more than the multiplicity of the line at infinity for $\omega$. This multiplicity is $\max(\deg P_1,\deg P_2)-\deg Q-k^{-1}\deg S$. If the line at infinity is not a pole of $G$, this bound does not apply but then $\deg U \leq \deg \hat{Q}+k^{-1}\deg S$. Thus in all cases we have
$$\deg U \leq \max(\max(\deg P_1,\deg P_2)-\deg Q+\deg \hat{Q}+1, \deg \hat{Q}+k^{-1}\deg S)$$
Now we want that $\omega-d(G)$ to have poles of order at most $1$ (possibly including infinity), and to have no singularity of order $>1$, i.e. its denominator should not have a common factor with $S$. Equation $Eq_1$ of step $2$ comes from this condition for finite poles, and equation $Eq_2$ for the line at infinity.

\end{proof}

A problem we will face is to determine if the integral of a reduced closed differential form $\omega=(P_1dx+P_2dy)/(QS^{1/2})$ can be written (see \cite{bronstein1998symbolic})
\begin{equation}\label{eqtor}
\int \omega= \lambda_0 \ln V ,\quad V\in\mathbb{C}(x,y,S^{1/2}).
\end{equation}

\begin{defi}
For an algebraic function $f\in S(x)^{-1/2}\mathbb{C}(x)$ with $S\in\mathbb{C}(x)$, we note $r=\hbox{res}_{x_0} f$ the residue of $f$ at $x=x_0$, which is given by the series expansion
$$f(x)=\sum\limits_{i=-3}^{l} \frac{a_i}{(x-x_0)^{i/2}}+\frac{r}{x-x_0}+o((x-x_0)^{-1})$$
Only the sign of $r$ depends on the branch choice of $\sqrt{S(x)}$ at $x_0$.
\end{defi}

\begin{lem}\label{lemtor} (see \cite{silverman1994advanced}, \cite{trager1984integration}, \cite{bertrand1995computing}, \cite{combot2021elementary})
Consider a reduced closed differential form $\omega=(P_1dx+P_2dy)/(QS^{1/2}), P_1,P_2,Q,S\in\mathbb{Q}[x,y]$ whose integral can be written as \eqref{eqtor}. Consider a straight line $y=z(x-x_0)+y_0$ where $(x_0,y_0)\in\mathbb{Q}^2$ is a regular point of $\omega$, and note $G=\tfrac{zP_1+P_1}{QS^{1/2}}(x,z(x-x_0)+y_0)$. We have
\begin{itemize}
\item There exists $d\in\mathbb{Q}$ such that all residues in $x$ of $G$ are in $\sqrt{d} \mathbb{Z}$.
\item For all $z$, the divisor defined by $D_z:\{(x,w)\in\mathbb{C}^2, w^2-S(x,z(x-x_0)+y_0)=0\} \rightarrow \mathbb{Z}$ $D(x,w)=d^{-1/2}\hbox{res}_{x,w} G$ is a torsion divisor.
\item The torsion order $N$ of $D_{z_0}$ is always the same for any $z_0$ such that $\Delta_x(QS)(z_0)\neq 0$.
\item The function $V$ of \eqref{eqtor} is such that $\hbox{div}_x(V(x,z(x-x_0)+y_0))=N D_z$ for almost all $z$.
\end{itemize}
\end{lem}

\begin{proof}
Looking at expression \eqref{eqtor} and restricting it to the line $y=z(x-x_0)+y_0$, all residues are integer multiples of $\lambda_0$, as all non zero residues lie outside the roots of $S$. Let us note
$$G=\frac{\tilde{P}}{\tilde{Q}\tilde{S}^{1/k}}=G$$
with $\tilde{Q},\tilde{S}$ coprime and $\tilde{P},\tilde{Q}$ coprime. The residues of $G$ are solutions of the polynomial
$$R=\hbox{resultant}(\tilde{P}^2-\lambda^2 (\partial_x\tilde{Q})^2\tilde{S},\tilde{Q},x)$$
The solutions in $\lambda$ of $R$ should not depend on $z$, as the residues of $G$ do not depend on $z$, $G$ being the $x$ derivative of the restriction of \eqref{eqtor} to the straight line $y=z(x-x_0)+y_0$. By construction, $R$ has rational coefficients and all its roots are integer multiples of $\lambda_0$. If $\lambda_0\notin \mathbb{Q}$, then by acting $\hbox{Gal}(\mathbb{Q}(\lambda_0):\mathbb{Q})$, this would also hold for the conjugates of $\lambda_0$, and thus in particular
$$\forall \sigma \in \hbox{Gal}(\mathbb{Q}(\lambda_0):\mathbb{Q}),\; \frac{\sigma(\lambda_0)}{\lambda_0}\in\mathbb{Q}.$$
Then considering the minimal polynomial of $\lambda_0$, its constant coefficient is thus a rational number times a power of $\lambda_0$, and thus $\lambda_0$ is a $k$-th root of a rational number. The conjugates of a $k$-th root are the multiples by $\xi$, a primitive $k$-th root of unity. Thus all such primitive $k$-th roots of unity should be rational, which leave the only possibilities $k=1,2$. Thus $\lambda_0$ is either rational, or the square root of a rational. This gives affirmation $1$.\\
Restricting \eqref{eqtor} to the line $y=z(x-x_0)+y_0$, we obtain a function $\tilde{F}$ whose root poles orders match the residues of $G$ up to some factor $\lambda_0$. As seen by affirmation $1$, we have $\lambda_0\in\sqrt{d} \mathbb{Q}$, and thus $d^{-1/2}\hbox{res}_{x,w} G$ match up to some rational factor the root poles orders of $\tilde{F}$. Thus divisor $D_z$ has a multiple which is principal, and thus is of torsion. This gives affirmation $2$.\\
According to Trager, a candidate for the order of a divisor can be obtained uniquely by using two good reduction primes $p,q$. Consider a $z_0$ such that $\Delta_x(QS)(z_0)\neq 0$, and two such good reduction primes for $D_{z_0}$. Then a unique candidate $N_{z_0}$ is found. We now consider the divisors $D_{z_0+mpq},\; m\in\mathbb{Z}$. We have $\Delta_x(QS)(z_0)\neq 0$ mod $p,q$ as $p,q$ are good reduction primes, and thus we have $\Delta_x(QS)(z_0+mpq)\neq 0$ mod $p,q$. By construction, after reduction modulo $p$ and $q$, we obtain for the reduction of $D_{z_0+mpq}$ exactly the same reduction as of $D_{z_0}$. Thus the unique candidate found $N_{z_0+mpq}=N_{z_0}$. If this candidate was indeed the torsion order, this means that $N_{z_0} D_z$ is principal for infinitely many $z$, and noting $N_{z_0} D_z=\hbox{div}(\tilde{F})$, as $\tilde{F}$ should be rational both in $x,z$, then this is true for all $z$. This gives affirmation $3$ and then $4$.
\end{proof}

The computation of the torsion order will be done using algorithm \cite{combot2021elementary}. For this it is necessary that a $z_0$ to be chosen such that $\Delta_x(QS)(z_0)\neq 0$ and that $\deg S(x,z_0(x-x_0)+y_0)$ to be odd. This last condition is not always fulfilled. For this, we introduce $\alpha$ a root of $S(x,z_0(x-x_0)+y_0)$, and consider the transformation $x \rightarrow \tfrac{1}{x}+\alpha$. Infinity is then a ramification point, and thus the new $S$ will be of odd degree. This transformation is a Moebius transformation, and thus does not change the torsion order of the divisor.

\begin{defi}
Consider an algebraic number $\alpha$, a field extension $\mathbb{L}$ of $\mathbb{Q}[\alpha]$. The trace of $\beta\in\mathbb{L}$ is
$$\hbox{tr}(\beta)= \frac{1}{\sharp \hbox{Gal}(\mathbb{L}:\mathbb{Q})}\sum\limits_{\sigma \in \hbox{Gal}(\mathbb{L}:\mathbb{Q})} \sigma(\beta)$$
The $\alpha$-trace of $\beta$ is  
$$\hbox{tr}_{\alpha}(\beta)= \frac{1}{\sharp \hbox{Gal}(\mathbb{L}:\mathbb{Q}[\alpha])} \hbox{coeff}_{\alpha}\left(\sum\limits_{\sigma \in \hbox{Gal}(\mathbb{L}:\mathbb{Q}(\alpha))} \sigma(\beta)\right)$$
where an element of $\mathbb{Q}[\alpha]$ is represented uniquely by a polynomial of degree less than the minimal polynomial of $\alpha$ over $\mathbb{Q}$.
\end{defi}

The trace and $\alpha$-trace are $\mathbb{Q}$ linear functions on $\mathbb{L}$. By convention, if $\alpha\in\mathbb{Q}$ the $\alpha$-trace will be the trace. Given $\beta \in \mathbb{L}$ with $P\in \mathbb{Q}[\alpha][x]$ unitary its minimal polynomial, the trace of $\beta$ over $\mathbb{Q}[\alpha]$ is minus the second coefficient of $P$ divided by $\deg P$. And thus the $\alpha$-trace is then the coefficient in $\alpha$ of this trace over $\mathbb{Q}[\alpha]$.\\

\noindent\underline{\sf ReduceDarbouxian}\\
\textsf{Input:} Vector field $X=(A,B),\; A,B\in\mathbb{Q}[x,y],\; A\wedge B=1$, and $k\in\mathbb{N}^*,\; F^k\in\mathbb{Q}(x,y)$ corresponding to a $k$-Darbouxian first integral of $X$.\\
\textsf{Output:} A rational first integral, or ``None''
\begin{enumerate}
\item Compute $E=\hbox{\underline{\sf HermiteReduction}}(-\frac{B}{A} F dx+ F dy)$. If $E\neq$ ``None'', non constant and $D(E)=0$ return $E^k$.
\item Note $(x_0,y_0)$ a non singular point of $-\frac{B}{A} F dx+ F dy$, and
$$\frac{\tilde{P}}{\tilde{Q}\tilde{S}^{1/k}}=\left(1-z\frac{B}{A}(x,z(x-x_0)+y_0)\right)F(x,z(x-x_0)+y_0)$$
with root multiplicities of $\tilde{S}$ in $x$ $<k$.
\item If $k=1$, do
\begin{enumerate}
\item If $E$ non constant return ``None''.
\item Compute $R=\hbox{resultant}_x(\tilde{P}-\lambda \partial_x\tilde{Q},\tilde{Q})\in\mathbb{Q}(z)[\lambda]$.
\item Note $\alpha$ a root of $R$, and factorize $Q=Q_1\dots Q_l$ in $\mathbb{Q}(\alpha)(z)[x]$.
\item For $i=1\dots l$, compute and assign
\begin{align*}
t_i:= & \hbox{resultant}_x(\tilde{P}-\lambda \partial_x\tilde{Q},Q_i)\in \mathbb{Q}[\alpha](z)[\lambda]\\
t_i:= & -(\hbox{leading coefficient of } t_i) /(\deg_\lambda t_i \times \hbox{second leading coefficient of } t_i)\\
t_i:= & \hbox{coefficient in }\alpha \hbox{ of } t_i
\end{align*}
\item Note
$$H=\left(\prod\limits_{i=1}^l Q_i^{d t_i}\right)_{\mid z=\frac{y-y_0}{x-x_0}} \quad \hbox{ where } d=\hbox{lcm}(\hbox{den}(t_i)_{i=1\dots l}).$$
If $D(H)=0$, return $H$, else return ``None''.
\end{enumerate}
\item Note $-\frac{B}{A} F dx+ F dy=(P_1dx+P_2dy)/(QS^{1/k})$. Search a non zero solution of the linear system
$$D\left(\frac{T}{\hbox{sf}(Q)}\right)+A\partial_y\left(\frac{B}{A}\right) \frac{T}{\hbox{sf}(Q)}=0$$
for $T\in\mathbb{Q}[x,y],\; \deg T < \deg \hbox{sf}(Q)$. If one found, return \underline{\sf ReduceDarbouxian}$(X,\frac{T}{\hbox{sf}(Q)})$.
\item Search for rational first integrals of numerator and denominator degree $\leq \max(0,6 (\deg \hbox{sf}(SQ)-1))$, if one found return it.
\item If $k\geq 3$, search a $1$-Darbouxian first integrals up to degree $\deg \hbox{sf}(Q)+(2\deg \hbox{sf}(S)+3\deg \hbox{sf}(Q))/(2k-4)$. If one found, note $\tilde{F}$ its $y$ derivative, and return \underline{\sf ReduceDarbouxian}$(X,\tilde{F})$.
\item If $k=2,3,4,6$ and $\tilde{Q}$ is constant and $\deg_x \tilde{P} -\deg_x \tilde{Q}-\tfrac{1}{k}\deg_x \tilde{S}<-1$, return ``Not handled''.
\item If $k\geq 3$, return ``None'', else $k=2$ and do
\begin{enumerate}
\item If $E=$``None'' or non constant, return ``None''.
\item Then $E$ is constant, compute
$$R(\lambda)=\hbox{resultant}_x(\tilde{P}^2-\lambda^2 (\partial_x \tilde{Q})^2 S,\tilde{Q})\left(\lambda^2-\left(\lim_{x=\infty} \frac{x\tilde{P}}{\tilde{Q}\tilde{S}^{1/2}}\right)^2 \right) \in\mathbb{Q}[\lambda]$$
\item If there exists $d,\; d^2\in\mathbb{Q}^*$ such that all solutions of $R$ are of the form $n_id,\; n_i\in\mathbb{Z}$ then build the divisor
$$\mathcal{D}_z= \sum_x \frac{1}{d}\hbox{res}_x \frac{\tilde{P}}{\tilde{Q}\tilde{S}^{1/2}}$$
else return ``None''.
\item Find a $z_0$ such that $\Delta_x(\tilde{Q}\tilde{S})\neq 0$.
\item If $\deg \tilde{S}(x,z_0)$ is even, note $\alpha$ a root of $\tilde{S}(x,z_0)$ and apply the Moebius transformation $\tfrac{1}{x}+\alpha$ to $\mathcal{D}_{z_0}$ giving a divisor $\tilde{\mathcal{D}}$. Else $\tilde{\mathcal{D}}=\mathcal{D}_{z_0}$.
\item Compute $N=\hbox{\underline{\sf TorsionOrder}}(\tilde{\mathcal{D}})$. If $N=0$, return ``None''.
\item Look for a function $H=H_1+H_2\sqrt{S},\; H_1,H_2\in\mathbb{Q}(\sqrt{d})(x,y)$ such that $\hbox{div}_x(H(x,z(x-x_0)+y_0))=N \mathcal{D}_z$. If none, return ``None''.
If $H_1$ not constant and $D(H_1)=0$ return $H_1$. If $H_2^2S$ not constant and $D(H_2^2S)=0$ return $H_2^2S$. Else return ``None''.
\end{enumerate}
\end{enumerate}

\begin{proof}[Proof of Theorem \ref{thm3}]
We begin with case $k=1$, which is specifically dealt in steps $1,2,3$.

\begin{lem}\label{lem5}
For $k=1$, the algorithm \underline{\sf ReduceDarbouxian} returns a rational first integral in steps $1-3$ if and only if one exists.
\end{lem}

\begin{proof}[Proof of the Lemma]
If $X$ admits a rational first integral, then
\begin{equation}\label{eq12}
\int -\frac{B}{A} F dx+ F dy= E(x,y)+\sum\limits_{i=1}^p \lambda_i \ln(F_i(x,y))=f(J(x,y)).
\end{equation}
Moreover, we can assume the $\lambda_i$ $\mathbb{Q}$ independent and note $\mathbb{L}=\mathbb{Q}(\lambda_1,\dots,\lambda_p)$. Differentiating both sides in $y$, we have
$$F(x,y)=\partial_y J(x,y) f'(J(x,y)) \;\Rightarrow f'(J(x,y))=\frac{F(x,y)}{\partial_y J(x,y)}$$
The right hand side is a rational function $\in\mathbb{C}(x,y)$, and is a function of $J$. As we can assume $J$ to be indecomposable, then the right hand side is a rational function of $J$. Thus $f'(z)\in\mathbb{C}(z)$ and so $f$ can be written
$$f(z)=g(z)+\sum\limits_{i=1}^q \mu_i \ln(f_i(z))$$
with $\mu_i$ $\mathbb{Q}$ independent. Now substituting $z=J(x,y)$ in this expression, we obtain $E(x,y)=g(J(x,y))$. Thus if $E$ is not constant, it should be itself a rational first integral. This case is tested in step $1$.
Else $E$ and $g$ are constant, and as the integrals are defined up to a constant, we can assume $E=0, g=0$.\\
Now the expression of the logarithmic part of \eqref{eq12} is not unique, as we can make a $\mathbb{Q}$-basis change of the $\lambda_i$. This implies that $p=q$ and there exists a matrix $M\in M_p(\mathbb{Q})$ such that
\begin{equation}\label{eq13}
\prod\limits_{i=1}^p f_j(J(x,y))^{M_{i,j}}= F_j(x,y).
\end{equation}
Thus all $F_j$ are function of $J$, and so should be first integrals. Let us consider polynomials $P_i\in\mathbb{L}[x,y]$ coprime and such that $F_j=\prod_{i=1}^l P_i^{A_{i,j}}$ with $A\in M_{l,p}(\mathbb{Q})$. Remark that we have the equivalence
$$\sum\limits_{i=1}^l \nu_i \ln P_i = \sum\limits_{i=1}^p \lambda_i \ln F_i \Leftrightarrow \nu A= \lambda$$
Let us apply the $\alpha$-trace to the right equality which gives $\hbox{tr}_\alpha (\nu) A =\hbox{tr}_\alpha (\lambda)$. Thus we also have
$$\sum\limits_{i=1}^l \hbox{tr}_\alpha(\nu_i) \ln P_i = \sum\limits_{i=1}^p \hbox{tr}_\alpha(\lambda_i) \ln F_i$$
We now choose for $\alpha=\nu_1$. Because of this choice, we know that $\hbox{tr}_\alpha(\nu)$ is not completely zero. Let us note $d$ the lcm of the denominators of the $\hbox{tr}_\alpha(\nu_i)$. As the $P_i$ are coprime and $\hbox{tr}_\alpha(\nu_1)=1$, we have that
$$H=\prod\limits_{i=1}^l P_i^{d\hbox{tr}_\alpha(\nu_i)}$$
is not constant (a factor of $P_1$ cannot be compensated in another $P_i$ as it cannot appear in another $P_i$). As it can also be written
$$H=\prod\limits_{i=1}^p F_i^{d\hbox{tr}_\alpha(\lambda_i)}$$
it is a rational first integral. This rational function is computed in steps $3(b)-(e)$. If this function $H$ happens not to be a first integral, then equation \eqref{eq12} does not hold and thus $X$ does not admit a rational first integral.
\end{proof}

Now case $k=1$ is done by Lemma \ref{lem5} and so we can assume $k\geq 2$. We write down the integral
$$\mathcal{F}(x,y)= \int \frac{P_1dx+P_2dy}{Q S^{1/k}}$$
with $P_1,P_2,Q,S\in\mathbb{Q}[x,y]$ and $k\in \mathbb{N}^*$ minimal. We can also assume that multiplicities of roots of $S$ are $<k$. If the system admits a rational first integral, then $\mathcal{F}(x,y)=f(J(x,y))$. Differentiating this relation, we get
$$\frac{P_1dx+P_2dy}{Q S^{1/k}}= f'(J(x,y)) dJ(x,y)$$
Thus $f'(J(x,y))^k$ is a rational function $G\in \mathbb{C}(x,y)$, and this $G$ is a function of $J(x,y)$ only. As we can assume that $J$ is indecomposable, then $G$ is a rational function of $J(x,y)$. Thus $f'$ is a $k$-th root of a rational function, and we will note
$$f(z)=\int \frac{\tilde{P}(z)}{\tilde{Q}(z)(\tilde{S}(z))^{1/k}} dz$$
with $\tilde{P},\tilde{Q},\tilde{S}$ polynomials. Remark that we know that the radical $(\tilde{S}(z))^{1/k}$ cannot simplified, as $J(x,y)$ is rational and the $k$-th root in the expression of $\mathcal{F}$ can only come from it. We can also assume $\tilde{S}$ has root multiplicities $<k$, and $\tilde{P}\wedge \tilde{Q}=1$. Applying a Moebius transformation, we can also assume that $\tilde{P}(z)/(\tilde{Q}(z)(\tilde{S}(z))^{1/k})$ is smooth at infinity, i.e. $\deg \tilde{P}-\deg \tilde{Q}-k^{-1} \deg \tilde{S}\leq -2$ and $k \mid \deg \tilde{S}$. We will now split our analysis depending on the number of roots of $\tilde{Q}$ and the genus $g$ of $z^k-\tilde{S}(x)$.\\

\textbf{When $\tilde{Q}$ has at least two distinct roots}. If $\tilde{Q}$ has two distinct roots $z_1,z_2$, then $f$ has a logarithmic singularity or a pole at $z_1,z_2$. And after composition by $J$, $\mathcal{F}$ has then logarithmic singularities or poles along the levels $F-z_1,F-z_2$. Applying a Moebius transformation to $J$ sending $z_1 \rightarrow 0, z_2 \rightarrow \infty$ gives another rational first integral $\tilde{J}$, whose numerator and denominator have factors which divide $Q$. Thus $\partial_y \ln J= R/\hbox{sf}(Q)$ with $\deg R<\deg \hbox{sf}(Q)$. This $1$-Darbouxian integral would be found in step $4$. Then a recursive call is made to \underline{\sf ReduceDarbouxian} on $T/V$, which is the $y$ partial derivative of the $1$-Darbouxian integral we found.

\begin{lem}\label{lem6b}
A vector field $X$ possessing a $k$-Darbouxian first integral (with $k\neq 1$) and a $1$-Darbouxian first integral possesses a rational first integral.
\end{lem}

\begin{proof}
For the $k$-Darbouxian first integral $\mathcal{F}$, we can write
$$\int \frac{P_1dx+P_2dy}{Q S^{1/k}}=g\left(\int Udx+Vdy \right)$$
for a univariate $g$. Differentiating both sides gives
$$\frac{P_1dx+P_2dy}{Q S^{1/k}}=g'\left(\int Udx+Vdy \right)(Udx+Vdy)$$
$$\frac{P_1}{QUS^{1/k}}=g'\left(\int Udx+Vdy \right),\quad \frac{P_2}{QVS^{1/k}}=g'\left(\int Udx+Vdy \right)$$
At least one of the two terms
$$\frac{P_1}{QUS^{1/k}},\frac{P_2}{QVS^{1/k}}$$
is non constant (as to be constant, they should be zero, and if both are zero then $\mathcal{F}$ is constant). Note $\mathcal{C}_h$ the $h$ level of this function, and up to exchange of variables, we can assume the first one.
Let us now assume that no rational first integral exists. Then $\int Udx+Vdy$ is non constant restricted to $\mathcal{C}_h$ for a generic $h$. Thus $g'(z)=h$ for $z$ on some open set, and thus $g$ is affine, and so we can write
$$\int \frac{P_1dx+P_2dy}{Q S^{1/k}}=a\left(\int Udx+Vdy\right)+b.$$
The Galois action of the $k$-th root on the left multiplies the left part by a $k$-th root of unity, and lets invariant the right part, thus $a=0$, which implies that $\mathcal{F}$ is constant. This is a contradiction, and thus a rational first integral exists.
\end{proof}

Using Lemma \ref{lem6b}, the recursive call will be made on a vector field possessing a $k$-Darbouxian first integral (with $k\neq 1$) and a $1$-Darbouxian first integral, and thus Lemma \ref{lem5} applies, and the recursive call terminates in steps $1-3$. Steps $1-3$ will find the rational first integral.\\

\textbf{When $g$ is at least $2$}. We have the relation
$$\frac{P_1dx+P_2dy}{Q S^{1/k}}= \frac{\tilde{P}(J(x,y)) dJ(x,y)}{\tilde{Q}(J(x,y))(\tilde{S}(J(x,y)))^{1/k}}$$
thus $(\tilde{S}(J(x,y)))^{1/k} \in S^{1/k}\mathbb{C}(x,y)$. This can be rewritten as the equation
$$SV^k=\tilde{S}(J),\quad V,J\in\mathbb{C}(x,y).$$
We now restrict ourselves to a generic straight line $y=z(x-x_0)+y_0$. We can build the rational application
\begin{align}
\varphi_z: & \{(x,w)\in\mathbb{C}^2,w^k=S(x,z(x-x_0)+y_0)\} \rightarrow \{(x,w)\in\mathbb{C}^2,w^k=\tilde{S}(x)\}\\
           & (x,w) \rightarrow (V(x,z(x-x_0)+y_0)w,J(x,z(x-x_0)+y_0))
\end{align}
Using the Riemann Hurwitz formula \cite{oort2016riemann}, we have that
$$2g'-2\geq \deg \varphi_z  (2g-2)$$
where the degree of $\varphi_z$ is the generic number of points of a fiber of $\varphi_z$ and $g'$ is the genus of $w^k=S(x,z(x-x_0)+y_0)$ for a generic $z$. Applying Lemma \ref{lem2}, we know that $g\geq k/12+1$. Using the genus formula for a superelliptic curve, we also have $g'\leq \frac{1}{2} ((k-1)(\deg_x \hbox{sf}(S)(x,z(x-x_0)+y_0)+1)-2k)+1$. Thus
$$\deg_x \varphi_z \leq \frac{((k-1)(\deg_x \hbox{sf}(S)(x,z(x-x_0)+y_0)+1)-2k)}{k/6} \leq 6 (\deg_x \hbox{sf}(S)(x,z(x-x_0)+y_0)-1)$$
Given $J(x,z(x-x_0)+y_0)$, there are at most $\max(\deg_x \hbox{num}(J(x,z(x-x_0)+y_0)),\deg_x \hbox{den}(J(x,z(x-x_0)+y_0)))$ solutions for $x$, and then $w$ can be uniquely recovered. Thus
$$\max(\deg_x \hbox{num}(J(x,z(x-x_0)+y_0)),\deg_x \hbox{den}(J(x,z(x-x_0)+y_0))) \leq 6 (\deg_x \hbox{sf}(S)(x,z(x-x_0)+y_0)-1).$$
Now for a generic $z$, numerator and denominator of $J(x,z(x-x_0)+y_0)$ will not simplify and the dominant term will not disappear, thus
$$\deg_x \hbox{num}(J(x,z(x-x_0)+y_0))= \deg \hbox{num}(J),\quad \deg_x \hbox{den}(J(x,z(x-x_0)+y_0))=\deg \hbox{den}(J),$$
$$\deg_x \hbox{sf}(S)(x,z(x-x_0)+y_0)=\deg \hbox{sf}(S)$$
and we get $\max( \deg \hbox{num}(J),\deg \hbox{den}(J)) \leq 6 (\deg \hbox{sf}(S)-1)$. This rational first integral would be found in step $5$.\\

\textbf{When $g=0$}. Using Lemma \ref{lem3}, the only possible expression for $f$ up to Moebius transformation would be
$$f(z)=\int \frac{\tilde{P}(z)}{\tilde{Q}(z) z^{l/k}} dz, l=1\dots k-1,\; l\wedge k=1$$
If $\tilde{Q}$ has at least two distinct roots, then the first case applies, so we can assume it has at most one, and up to Moebius transformation, we can put it at $1$
$$f(z)=\int \frac{\tilde{P}(z)}{(z-1)^{\tau} z^{l/k}} dz.$$
Now if $\deg \tilde{P} \geq \tau$, $f$ has a pole at infinity. For $\tau=0$, we would get then integral
$$f(z)=\int \frac{\tilde{P}(z)}{z^{l/k}} dz \in \mathbb{C}(z^{1/k})$$
Thus $f(J)$ would be algebraic, and would then equal the Hermite reduction $E$ in step $1$. Then $E^k$ would indeed be a rational first integral returned in step $1$.

When $\tau>0$, up to Moebius transformation, the previous case (two singularities) would apply. So the only possibility left is
$$f(z)=\int \frac{\tilde{P}(z)}{(z-1)^\tau z^{l/k}} dz,\; \tau \in\mathbb{N}^*, \deg P \leq \tau -1, l=1\dots k-1,\; l\wedge k=1.$$
Noting $J=R/U+1$, the curve $R=0$ will be a pole of $f(J(x,y))$, and thus its factors will divide $\hbox{sf}(Q)$ and so its square free part will have degree less than $\deg \hbox{sf}(Q)$. Differentiating the relation $\mathcal{F}(x,y)=f(J(x,y))$, we find that there exists $V\in\mathbb{C}(x,y)$ such that $SV^k=J^l$. Thus
$$ SV^k=\left(\frac{R}{U}+1\right)^l$$
$$ S(VU^l)^k=(R+U)^lU^{(k-1)l}$$
$$ SW^k=R(R+U)^{l-1}U^{(k-1)l}+(R+U)^{l-1}U^{(k-1)l+1}  ,\;\hbox{ noting } W=VU^l\in\mathbb{C}[x,y]$$
Let us note
$$\deg W=w,\;\deg R=r,\; \deg U=u,\; \deg \hbox{sf}(Q)=q,\; \deg \hbox{sf}(S)=s.$$
Let us first consider the case $k\geq 3$. We have that
$$\hbox{sf}(SW^kR(R+U)^{l-1}U^{(k-1)l+1})=\hbox{sf}(SWRU)$$
as $(R+U)^{l-1}$ should be a factor of $SW^k$. Applying Mason's $ABC$ theorem \cite{mason1984diophantine}, we have
\begin{equation}\label{eqmas1}
\max(s+kw,r+(k-1)u,ku)\leq s+w+q+u 
\end{equation}
remembering that the factors of $R$ are factors of $Q$, and that $l\geq 1$. Adding these $3$ inequalities, we obtain
$$s+kw+r+(2k-1)u\leq 3(s+w+q+u)$$
$$\Leftrightarrow \;\;\; (k-3)w+(2k-4)u\leq 2s+3q$$
$$ \Leftrightarrow u\leq \frac{2s+3q}{2k-4}. $$
We have that $U/R$ is a rational first integral, and as factors of $R$ are factors of $Q$ then $\partial_y \ln (U/R)$ will be of degree at most $\frac{2s+3q}{2k-4}+q$. This $1$-Darbouxian first integral would be detected in step $6$, and then the recursive call would terminate in step $3$ according to Lemma \ref{lem5} and return a rational first integral according to Lemma \ref{lem6b}.\\

We now consider the case $k=2$, and thus $l=1$. The integral 
$$f(z)=\int \frac{\tilde{P}(z)}{(z-1)^{\tau} \sqrt{z}} dz=\tilde{G}+\sum_{i=1}^l \lambda_i \ln \tilde{F}_i(\sqrt{z})$$
is always elementary, and thus $\mathcal{F}$ should be elementary. Thus $\mathcal{F}$ can also be written
$$\mathcal{F}(x,y)=E(x,y)+\sum_{i=1}^{l'} \lambda'_i \ln F_i(x,y)$$
As $\mathcal{F}(x,y)=f(J(x,y))$, by identification we have $\tilde{G}(J(x,y))=E(x,y)$, and if $E$ is non constant, then $E^k$ should be a rational first integral, which would have been detected in step $1$. So if $E$ does not exist or is non constant and not a first integral, we cannot be in this case. As it is the last one studied by the algorithm, this would imply that no rational first integral exists. If $E$ is constant, as $\mathcal{F}$ is defined up to addition of a constant, we can then assume $E=0$. We also then have that function $f$ has no poles of order $\geq 2$, and thus $\tau=1$ and so
$$f(z)=\int \frac{d}{(z-1) \sqrt{z}} dz=a \ln\left( \frac{1-\sqrt{z}}{1+\sqrt{z}} \right),\; a\in\mathbb{C}^*$$
Thus after composition by $J$, we face the problem \eqref{eqtor}. The residues of $\omega$ restricted to the line $y=z(x-x_0)+y_0$ are computed in step 8(b) where $R$ is a polynomial whose roots in $\lambda$ are the residues. Using Lemma \ref{lemtor}, we know that polynomial $R$ should have roots in $d\mathbb{Z}$ for some common $d$ with $d^2\in\mathbb{Q}$, and this is tested in step 8(c). We then build the divisor $\mathcal{D}_z$, and using Lemma \ref{lemtor}, we know that $\mathcal{D}_z$ is of torsion and that $\mathcal{D}_{z_0}$ is also a torsion divisor for the $z_0$ chosen in step 8(d). In step 8(e),8(f), we compute a torsion order candidate for $\mathcal{D}_{z_0}$. Step 8(e) is necessary to apply the algorithm \underline{TorsionOrder} from \cite{combot2021elementary} , and such Moebius transformation does not change the torsion order. If the candidate found is $N=0$, then $\mathcal{D}_{z_0}$ is not of torsion, and thus $\mathcal{D}_z$ is neither and thus $\int\omega$ is not elementary, and thus there are no rational first integrals. Else, using Lemma \ref{lemtor}, we know that this $N$ is the unique possible candidate for the torsion order of $\mathcal{D}_z$. In step 8(g), we look for a rational function on the curve, which can be written $H_1+H_2\sqrt{S}$ whose divisor is $N\mathcal{D}_z$. If none is found, then $\mathcal{D}_z$ is not of torsion and thus there are no rational first integrals. Else we know that $\int \omega- \frac{d}{N}\ln(H_1+\sqrt{S}H_2)$ is an integral of the first kind (as there are no singularity on the surface $w^2=S(x,y)$), and thus $\int \omega$ is elementary if and only if this first kind integral is zero. Then $H_1+H_2\sqrt{S}$ would be an algebraic first integral and thus $H_1,H_2^2S$ would be rational first integrals. At least one of them is not constant, and thus should give a non constant rational first integral. If not, then $\int \omega$ was not elementary, and thus there was no rational first integral.\\

\textbf{When $g=1$ and $\tilde{Q}$ has one root}. By Moebius transformation, we can send this pole to infinity. Now either this is a ramification point (which will give elliptic integrals of the second kind), or infinity is not ramified. Using Lemma \ref{lem3}, we have the superelliptic curves of genus $1$ up to Moebius transformation, and this gives for infinity ramified the possibilities for $f$
$$\int \frac{\tilde{P}(z)}{ \sqrt{z(z-1)(z-a)}} dz,\;\; \int \frac{\tilde{P}(z)}{ (z(z-1))^{1/3}} dz,\;\; \int \frac{\tilde{P}(z)}{ (z(z-1))^{2/3}} dz$$
$$\int \frac{\tilde{P}(z)}{ (z^2(z-1))^{1/4}} dz,\;\; \int \frac{\tilde{P}(z)}{ (z^2(z-1)^3)^{1/4}} dz,\;\; \int \frac{\tilde{P}(z)}{ (z^3(z-1)^2)^{1/6}} dz,\;\; \int \frac{\tilde{P}(z)}{ (z^3(z-1)^4)^{1/6}} dz$$
and for infinity not ramified
$$\int \frac{\tilde{P}(z)}{ \sqrt{z(z-1)(z-a)(z-b)}} dz,\;\; \int \frac{\tilde{P}(z)}{ (z(z-1)(z-a))^{1/3}} dz,\;\; \int \frac{\tilde{P}(z)}{ (z(z-1)(z-a))^{2/3}} dz$$
$$\int \frac{\tilde{P}(z)}{ (z^2(z-1)(z-a))^{1/4}} dz,\;\; \int \frac{\tilde{P}(z)}{ (z^2(z-1)^3(z-a)^3)^{1/4}} dz,$$
$$\int \frac{\tilde{P}(z)}{ (z^3(z-1)^2(z-a))^{1/6}} dz,\;\; \int \frac{\tilde{P}(z)}{ (z^3(z-1)^4(z-a)^5)^{1/6}} dz.$$
Now composing $f$ with $J$, the poles of $J$ will become singular curves for $f$ (poles in the second kind case, logarithmic singularities in the third kind case). Thus the factors of the denominator of $J$ are factors of $Q$. The function $J$ defines a morphism $\varphi_y$ from the curve $z^k-S(x,y)$ to an elliptic curve, and thus is corresponds to a rational solution of one of these Diophantine equations
$$SV^2=J(J-1)(J-a),\; SV^3=J(J-1),\; SV^3=J^2(J-1)^2,\; SV^4=J^2(J-1),$$
$$SV^4=J^2(J-1)^3,\; SV^6=J^3(J-1)^2,\; SV^6=J^3(J-1)^4,\; SV^3=J(J-1)(J-a),$$
$$SV^2=J(J-1)(J-a)(J-b),\; SV^3=J^2(J-1)^2(J-a)^2,\; SV^4=J^2(J-1)(J-a),$$
$$SV^4=J^2(J-1)^3(J-a)^3,\; SV^6=J^3(J-1)^2(J-a),\; SV^6=J^3(J-1)^4(J-a)^5.$$
We will now use Lemma \ref{lem1} to solve these equations in $\mathbb{C}[x,y,(SQ)^{-1}]$, which gives the bound
$$ \max( \deg \hbox{numer}(J),\deg \hbox{denom}(J)) \leq 6 \max(0,\deg \hbox{sf}(SQ)-1)$$
This first integral would be found in step $5$.\\

\textbf{When $g=1$ and $\tilde{Q}$ is constant}. According to Lemma \ref{lem3}, for genus $1$ we need $k\in \{2,3,4,6\}$. As $f$ has no pole on the elliptic curve, after composition by $J$, the function $\mathcal{F}$ has no pole on the surface $z^k=S(x,y)$ (even along the line at infinity). This requires that
$$\mathcal{F}(x,y)=\int \frac{Bdx-Ady}{W^{1/k}},\quad W\in\mathbb{C}[x,y],\; k^{-1}\deg W>\max(\deg A,\deg B)+1$$
and factors of $W$ have multiplicities $\leq k-1$. This is condition tested in step $7$. If this is the case, then we should search decompositions of the form
$$\int^{J(x,y)} \frac{1}{\sqrt{z(z-1)(z-a)}} dz\quad k=2,\qquad \int^{J(x,y)} \frac{1}{(z(z-1))^{2/3}} dz\quad k=3$$
$$\int^{J(x,y)} \frac{1}{(z^2(z-1)^3)^{1/4}} dz,\quad k=4,\qquad \int^{J(x,y)} \frac{1}{(z^3(z-1)^4)^{1/6}} dz,\quad k=6$$
In this situation, the algorithm \underline{ReduceDarbouxian} returns in step $7$ ``Not handled''.
\end{proof}

Remark that the returned rational first integral is not always in $\mathbb{Q}(x,y)$. First integrals returned in step $3$ and step $8(g)$ can have algebraic coefficients. Noting $\mathbb{Q}[\alpha]$ the field of the coefficients, it is however always possible to consider Newton sums
$$S_j=\sum\limits_{\sigma\in\hbox{Gal}(\mathbb{Q}[\alpha]:\mathbb{Q})} \alpha^j\sigma(J),\quad j\in\mathbb{N}^*,$$
If $J$ is a rational first integral, $\sigma(J)$ is also a rational first integral for all $\sigma\in\hbox{Gal}(\mathbb{Q}[\alpha]:\mathbb{Q})$, and thus $S_j$ is a rational first integral (possibly constant). However, if $S_j$ is constant for all $j$, then as a Vandermonde matrix on the conjugates of $\alpha$ is invertible, we would have $j$ constant which is impossible. Thus a non constant first integral $S_j \in\mathbb{Q}(x,y)$ always exists.

Remark that building this rational first integral with coefficients in $\mathbb{Q}$ can in some way be much more complicated than the one with coefficients in $\bar{\mathbb{Q}}$. Indeed, even though the degree of the minimal rational first integral is not higher than the one with coefficients in $\bar{\mathbb{Q}}$ (see \cite{cheze2014decomposition}), the factors structure can change as in the following example
$$J=\frac{(x+\sqrt{2})^3(x+y\sqrt{2})^5}{(x-\sqrt{2})^3(x-y\sqrt{2})^5}$$
The field $\mathbb{C}(J)$ can also be generated by the rational function
\begin{footnotesize}
$$\frac{(5x^6y+20x^4y^3+4x^2y^5+3x^6+60x^4y^2+60x^2y^4+30x^4y+120x^2y^3+24y^5+2x^4+40x^2y^2+40y^4)x}
{x^8+20x^6y^2+20x^4y^4+30x^6y+120x^4y^3+24x^2y^5+6x^6+120x^4y^2+120x^2y^4+20x^4y+80x^2y^3+16y^5}$$
\end{footnotesize}
This function is of the same degree, but does not look so pleasant as the first one as no high multiplicity factors appears.

\section{The resistant case}

If we arrive in the not handled case, then the first integral
$$\mathcal{F}(x,y)=\int \frac{P_1dx+P_2dy}{S^{1/k}}$$
has no pole, and we should find the possible decompositions with an elliptic integral. If $\mathcal{F}=f(J)$ where $J$ is a rational first integral, differentiating both sides gives a Diophantine equation of the form
\begin{equation}\label{eqmor}
J(J-1)(J-a)=SV^2,\quad J^2(J-1)^2=SV^3,\quad J^2(J-1)^3=SV^4,\quad J^3(J-1)^4=SV^6.
\end{equation}
These equations are elliptic equations, and the Mordell Weil Theorem \cite{silverman1994advanced} applies, so the rational solutions are finitely generated (for any fixed $a$ in the first case). If $(J,V)$ is a torsion solution, then for some $N\in\mathbb{N}^*$ we would have respectively
$$N\int^{J(x,y)} \frac{1}{\sqrt{z(z-1)(z-a)}} dz,\; N\int^{J(x,y)} \frac{1}{(z(z-1))^{2/3}} dz,$$
$$N\int^{J(x,y)} \frac{1}{(z^2(z-1)^3)^{1/4}} dz,\; N\int^{J(x,y)} \frac{1}{(z^3(z-1)^4)^{1/6}} dz$$
constant, and thus $J$ would be constant. So the torsion solutions are not interesting for us.

In the case $k=2$, a free parameter $a$ appears, which can be recovered using the $j$ invariant of the elliptic curve. For a basis $\gamma_1,\dots,\gamma_p$ of the homotopy group of the surface $w^k=S(x,y)$, we can compute the monodromy of the integral $\mathcal{F}$ along these cycles, which give complex numbers $c_1,\dots,c_p$. If a decomposition with an elliptic integral is possible, then $c_1,\dots,c_p$ should define to a lattice $\mathcal{R}$ of $\mathbb{C}$. Thus taking two non colinear $c_i,c_j$, the lattice $\tilde{\mathcal{R}}$ generated by $c_i,c_j$ will be $\subset \mathcal{R}$ and there will exist $n\in\mathbb{N}^*$ such that $\mathcal{R} \subset \tfrac{1}{n} \tilde{\mathcal{R}}$. We can thus know $\mathcal{R}$ up to a factor, and then obtain the $j$ invariant from $c_i,c_j$. This defines uniquely a $j$ invariant candidate, although in practice we can only compute it numerically because the numbers $c_i,c_j$ are not algebraic. For $k\in\{3,4,6\}$, the lattice $\mathcal{R}$ is fixed up to dilatation, and thus the $j$ invariant is known in advance, respectively $j=1728,0,1728$.

Although the Morweill Weil group is infinite, the function $f$ is a morphism which transforms the addition law of the group to the classical addition. Thus knowing a basis $(J_1,V_1),\dots,(J_r,V_r)$ of the group, we then know that $\mathcal{F}$ should be written
$$\mathcal{F}(x,y)=\sum\limits_{i=1}^r \lambda_i f(J_i).$$
Then if the $\lambda_i$ are up to a common factor in
\begin{itemize}
\item $\mathbb{Q}$ for $j\neq 0,1728$
\item $\mathbb{Q}(i)$ for $j=0$
\item $\mathbb{Q}(e^{2i\pi/3})$ for $j=1728$
\end{itemize}
we can regroup the terms with addition law (and special multiplication law in cases $j=0,1728$) to an expression $\lambda_0 f(J)$ where $J$ will thus be a rational first integral.

To compute the Mordell Weil group, the difficult point is typically to find the rank $r$ (see an example in \cite{hulek2011calculating}). Once known, an exhaustive search can find the generators, however, we typically only obtain upper bounds, which if not tight, does not allow to conclude. Let us remark that along a generic line $y=z(x-x_0)+y_0$, any solution $(J(x,z(x-x_0)+y_0),V(x,z(x-x_0)+y_0))$ defines a morphism from the superelliptic curve $\mathcal{C}_z=\{w^k=S(x,z(x-x_0)+y_0)\}$ to an elliptic curve (and always the same as the possible $j$ invariant is unique). Now considering the Jacobian $\hbox{Jac}(\mathcal{C}_z)$, such morphism induces an elliptic factor in this Jacobian. Thus each independent solution $(J,V)$ of the Mordell Weil group will define an elliptic factor of $\hbox{Jac}(\mathcal{C}_z)$, and with the same $j$ invariant. Thus the rank is bounded by the largest power of an elliptic curve in $\hbox{Jac}(\mathcal{C}_z)$.

Remark that in the case $k=2$, $\hbox{Jac}(\mathcal{C}_z)$ could have different elliptic factors, but only one would correspond to the $j$ invariant found with monodromy numbers $c_i,c_j$. For $k=3,4,6$, there cannot be elliptic factors for $j\neq 1728,0,1728$ respectively, thus only one curve has to be considered. We will focus on these cases which are easier.\\

\noindent
\textbf{Example with many elliptic factors}\\
Consider the superelliptic curve
$$w^3=u(u-1)(5u^2-5u+8)$$
This curve is of genus $3$, and its Jacobian decomposes in $3$ elliptic factors, giving rise to the following formulas
$$\int^{\frac{(u+4)^3(u-1)^2}{(5u^2-5u+8)^2)}} \frac{1}{(s(s-1))^{2/3}} ds=-5\int 20^{1/3}\frac{1}{w} du$$
$$\int^{\frac{(i\sqrt{15}-10u+5)^3}{20(3i\sqrt{15}+10u-5)}} \frac{1}{(s(s-1))^{2/3}} ds =-10 \int 2^{1/3}\frac{2u-1+i\sqrt{15}}{w^2} du$$
$$\int^{\frac{(-i\sqrt{15}-10u+5)^3}{20(-3i\sqrt{15}+10u-5)}} \frac{1}{(s(s-1))^{2/3}} ds =-10\int 2^{1/3}\frac{2u-1-i\sqrt{15}}{w^2} du$$
Now if this curve appears as a section $\mathcal{C}_z$, we should be able to deform it continuously and still keeping elliptic factors. If we ask to keep all $3$, then the only possible deformation should give a family of isomorphic curves (as their Jacobian would be isomorphic). This gives a first integral of the form
$$\mathcal{F}(x,y)=\int^{h_y(x)} \frac{1}{(u(u-1)(5u^2-5u+8))^{2/3}} du$$
where $h$ is a Moebius transformation in $x$ depending rationally in $y$. We then expect that this case should fall in the ``Not handled'' case as it can be written with an elliptic integral
$$\mathcal{F}(x,y)=\frac{1}{20 2^{1/3}} \int^{\frac{5}{8}h_y(x)^2-\frac{8}{8} h_y(x)+1} \frac{1}{(s(s-1))^{2/3}} ds $$
however, the previous morphism to the superelliptic curve would be found in previous step, and thus algorithm would return $h_y(x)$.\\

\noindent
\textbf{Example with non trivial deformation}\\
We need a non trivial deformation to avoid having a morphism to a superelliptic curve, but still keep an elliptic factor. If there are $g-1$ elliptic factors, then by quotient of the Jacobian will have in fact $g$ factors, and thus the Jacobian is fixed. Thus we need to have at most $g-2$ elliptic factors, and so the smallest interesting genus is $3$. Consider the curve
$$w^3=(2y^3x^3+2x+2)(x+1)$$
where $y$ is a parameter. For almost any value of $y$, we have exactly one elliptic factor, giving rise to the formula
$$\int^{-2\frac{y^3x^3}{x+1}} \frac{1}{(s(s-1))^{2/3}} ds=\int \frac{-2^{1/3} y(2x+3)}{w^2} du$$
Now differentiating the left hand side in $x$ and $y$, we obtain the differential form
$$\frac{-2^{1/3} y(2x+3)}{((2y^3x^3+2x+2)(x+1))^{2/3}} dx -3\frac{2^{1/3} x(x+1)}{((2y^3x^3+2x+2)(x+1))^{2/3}} dy$$
This defines a $3$-Darbouxian first integral, and the corresponding vector field is
$$\dot{x}=3x(x+1),\; \dot{y}=y(2x+3).$$
This vector field admits however a rational first integral $y^3x^3/(x+1)$ which is of degree $6$, and thus smaller than the bound in step $5$ which is $6\times (7-1)=36$, and thus would be found even if no morphisms to a superelliptic curve $g>1$ exists.\\

\noindent
\textbf{Example with non trivial deformation and double elliptic factor (see \cite{kloosterman2008elliptic} for a comparable example)}\\
To obtain arbitrary large degree rational first integral, we need to combine at least two elliptic integrals and thus need at least $2$ elliptic factors, and so with sections of genus at least $4$. To build such an example, we consider functions $J$ of the form
$$J=\frac{r(x+a_0)^3}{x^2+xb_1+b_0}$$
such that $J(J-1)$ is up to cubes a polynomial of the form $x(x-1)(x^3+c_2x^2+c_1x+c_0)$ (this form being chosen to fix the representation of a genus $4$ curve up to Moebius transformation). This gives for the transformation and the curve
$$J=\frac{r(x+a_0)^3}{3rxa_0^2+ra_0^3+3rxa_0+rx+x^2-x}$$
$$w^3=rx(x-1)(rx+3ra_0+r-1)(3rxa_0^2+ra_0^3+3rxa_0+rx+x^2-x)$$
Now this is a two parameter family of genus $4$ curves with one elliptic factor, and we compute the self intersection of this surface. This self intersection consists of $3$ $\mathbb{Q}$-irreducible components of dimension $1$, two of them being genus $1$ curves, and one $\bar{\mathbb{Q}}$ reducible of genus $0$
$$-6ra_0^2-6ra_0-r+2a_0+1\pm i\sqrt{3} (2ra_0+r-1)=0$$
Thus on this curve in the parameter space, there are two different $J$ leading to the same genus $4$ curve, which will then have a double elliptic factor. Making a parameter change $a_0=y(1-i\sqrt{3})/(i\sqrt{3}-2y+1)$, this gives
$$J_1,J_2=\frac{(i\sqrt{3}(y-x)+2yx-y-x)^3}{8(y^2x-yx-y+x)(y^2x-y^2-yx+x)},\frac{(-i\sqrt{3}(y-x)+2yx-y-x)^3}{8(y^2x-yx-y+x)(y^2x-y^2-yx+x)}$$
$$w^3=(y^2x-yx-y+x)(y-x)(y^2x-y^2-yx+x)x(x-1)$$
We can now consider a linear combination of the two elliptic integrals
$$\mathcal{F}(x,y)=(1+i\sqrt{3}\alpha) \int^{J_1} \frac{dz}{(z(z-1))^{2/3}}+(1-i\sqrt{3}\alpha) \int^{J_2} \frac{dz}{(z(z-1))^{2/3}}$$
\begin{footnotesize}
$$=\alpha \int \frac{((3x^2y^2-6xy^3-3x^2y+6xy^2+3y^3+3x^2-6xy)dx-3x(x-1)(x-2xy+y^2)dy}{((y^2x-yx-y+x)(y-x)(y^2x-y^2-yx+x)x(x-1))^{2/3}}+ $$
$$\int \frac{(2y^3x^2-2y^3x-3y^2x^2-y^3+2y^2x+3yx^2+2y^2-2yx-x^2) dx-x(x-1)(2xy^2-2xy-y^2+5x-4y)dy}{((y^2x-yx-y+x)(y-x)(y^2x-y^2-yx+x)x(x-1))^{2/3}}$$
\end{footnotesize}
Thus the vector field
$$\dot{x}=-x(x-1)(6\alpha xy-3\alpha y^2-2xy^2-3\alpha x+2xy+y^2-5x+4y)$$
$$\dot{y}=((3y^2-3y+3)x^2-6y(y^2-y+1)x+3y^3)\alpha+(2y^3-3y^2+3y-1)x^2-2(y^3-y^2+y)x-y^3-2y^2$$
admits a rational first integral if and only if
$$\frac{1+i\sqrt{3}\alpha}{1-i\sqrt{3}\alpha}\in \mathbb{Q}\left(e^{2i\pi/3}\right) \Leftrightarrow \alpha \in \mathbb{Q}(i\sqrt{3}).$$
This first integral can be of arbitrary degree, and thus once it is of degree $> 6\times (9-1)=48$, the algorithm \underline{ReduceDarbouxian} will reach step $7$ where it will return ``Not handled''.\\

\noindent
\textbf{A Lins Neto example (see \cite{neto2002some}, \cite{guillot2002exemples})}\\
A similar example was found by Lins Neto with a $6$-Darbouxian first integral
$$\dot{x}=y^2+4x-9x^2+2(2\alpha+1)y(1-2x), \; \dot{y}=3(2\alpha+1)(x^2-y^2)+6y(1-2x)$$
This vector field admits a $6$-Darbouxian first integral given by the formula $\mathcal{F}(x,y)=$
$$\int \frac{(9x^2+4(2\alpha+1)xy-y^2-4x-2(2\alpha+1)y)dx+(9x^2+60xy-y^2-4x-30y)dy}{(9x^4+6x^2y^2+y^4-4x^3-12xy^2+4y^2)^{5/6}}$$
and depends on a parameter $\alpha$. The closed $1$-form has no poles as the denominator is $(9x^4+6x^2y^2+y^4-4x^3-12xy^2+4y^2)^{5/6}$, and thus fits the case ``Not handled''. To look for rational first integrals, we need that $I(x,y)=f(J(x,y))$ with $J\in\mathbb{C}(x,y)$ and $f$ a superelliptic integral. If the superelliptic integral is of genus $\geq 2$, then we should find a rational first integral of degree $\leq 6\times (4-1)=18$. Beyond this bound, \underline{ReduceDarbouxian} would not find the rational first integral in step $6$ (nor a Darbouxian first integral of degree $\leq 1$ in step $5$), and thus would reach step $7$ where it would return ``Not handled'', due to the possibility of $f$ being an elliptic integral. Using Lemma \ref{lem3}, the possible elliptic integral for $f$ is
$$f(z)=\int \frac{1}{(z^3(z-1)^4)^{1/6}} dz$$
It appears that $\mathcal{F}$ can in fact be written $\mathcal{F}(x,y)=$
$$\frac{1}{12}\int^{\frac{4(y-x)^3}{9x^4+6x^2y^2+y^4-4x^3-12xy^2+4y^2}} \!\!\!\!\!\! \frac{dz}{((z-1)^3z^4)^{1/6}}-\frac{\alpha}{12}\int^{\frac{4(2x-1)^3}{9x^4+6x^2y^2+y^4-4x^3-12xy^2+4y^2}} \!\!\!\!\!\!\frac{ dz}{((z-1)^3z^4)^{1/6}}$$
The existence of such expression comes from two elliptic factors in the Jacobian of $w^6=9x^4+6x^2y^2+y^4-4x^3-12xy^2+4y^2$ ($y$ seen as a parameter), which is a curve of genus $7$. As in the previous example, to be able to regroup the two integrals, we need $\alpha\in \mathbb{Q}(e^{2i\pi/3})$ and this is the condition for having a rational first integral.

\section{Conclusion}

These algorithms are implemented in Maple 2021 and are available on my web page http://combot.perso.math.cnrs.fr/software.html. For practical computations, the Liouvillian reduction algorithm is the fastest as it does not require the computation of rational first integrals up to some large bound. This is typically the most costly part of the algorithms \underline{ReduceRiccati} and \underline{ReduceDarbouxian}. The question of the size of the output depending on the size of the input stays open.  For Liouvillian reduction, the rational or Darbouxian first integral returned always has either small degree of structure which guarantees a compact representation of the output. For Riccati  reduction, the Liouvillian solutions obtained by Kovacic algorithm can be of arbitrary larger degree, and their degree depend on the local exponents at singularities. Except when there exists a pullback to a (solvable) hypergeometric equation, it is unclear if it would be possible to present them in a compact representation, i.e. a representation whose size would depend only on the degree of the input and not the local exponents. For Darbouxian reduction, two difficult cases appear. First the algorithm is not complete, as the elliptic pullback case is left aside. The example built by Lins Neto has a $6$-Darbouxian first integral which is a linear combination of elliptic pullbacks. In this case, we can always represent the rational first integrals as sums and multiplications on an elliptic curve on finitely many rational functions. However these base functions are a generator set of the Mordell Weil group of an elliptic threefold, and it is unknown if we can bound them in function of the degree of the threefold. The second difficult situation is that for a $2$-Darbouxian first integral, we face at some point the need to compute the torsion order of a divisor in the Jacobian of a hyperelliptic curve with coefficients in a quadratic extension of $\mathbb{Q}$. Although a bound is known for elliptic curves \cite{7}, in larger genus the existence of such bound is unclear \cite{6}. Thus even fixing the degree and the residues, the torsion order could become arbitrary large, and thus the possible rational first integral could become of arbitrary large degree. Thus, we cannot rule out the existence of a sequence of vector fields, with fixed degree, with fixed eigenvalues at singular points, admitting elementary $2$-Darbouxian first integrals of fixed degree, but admitting an arbitrary large degree minimal rational first integral.

\label{}
\bibliographystyle{plain}
\bibliography{reduction}

\end{document}